\documentclass[a4paper,reqno]{amsart}
\usepackage{microtype}
\usepackage{amssymb}
\usepackage{amsfonts}
\usepackage{geometry}
\usepackage{hyperref}

\newtheorem{theorem}{Theorem}
\theoremstyle{plain}

\newtheorem{corollary}{Corollary}

\newtheorem{definition}{Definition}
\newtheorem{example}{Example}

\newtheorem{lemma}{Lemma}

\numberwithin{equation}{section}
\numberwithin{theorem}{section}  
\numberwithin{proposition}{section}  
\numberwithin{lemma}{section}  
\numberwithin{corollary}{section}

\input{tcilatex}

\begin{document}
\title[Velocity Overshoot Criterion of a Body in a Kinetic Gas]{Velocity
Reversal Criterion of a Body Immersed in a Sea of Particles}
\author{Xuwen Chen}
\address{Department of Mathematics, Brown University, Providence, RI 02912}
\email{chenxuwen@math.brown.edu}
\urladdr{http://www.math.brown.edu/\symbol{126}chenxuwen/}
\author{Walter Strauss}
\address{Department of Mathematics and Lefschetz Center for Dynamical
Systems, Brown University, Providence, RI 02912}
\email{wstrauss@math.brown.edu}
\urladdr{http://www.math.brown.edu/\symbol{126}wstrauss/}
\date{v2, 11/19/2014}

\begin{abstract}
We consider a rigid body colliding with a continuum of particles. We assume
that the body is moving at a velocity close to an equilibrium velocity $%
V_{\infty }$ and that the particles colliding with the body reflect
diffusely, that is, probabilistically with some probablility distribution $K$%
. We find a condition that is sufficient and almost necessary that the
collective force of the colliding particles reverses the relative velocity $%
V(t)$ of the body, that is, changes the sign of $V(t)-V_{\infty }$, before
the body approaches equilibrium. Examples of both reversal and irreversal
are given. This is in strong contrast with the pure specular reflection case
in which only reversal happens.
\end{abstract}

\maketitle

\section{Introduction}

The problem that we are considering has a free boundary, the location of the
body. The other unknown is the configuration of the particles. The particles
may collide with the body elastically or diffusely. Boundary interactions in
kinetic theory are very poorly understood, even when the boundaries are
fixed. Free boundaries are even more difficult. For this reason we have
chosen to consider only the \textit{simplest} problem of this type, namely,
we assume the particles are identical and are rarefied, that is, do not
interact among themselves but only with the body. We assume that the whole
system, consisting of the body and the particles, starts out rather close to
an equilibrium state.

We consider classical particles that are extremely numerous. While one could
consider modeling them as a fluid, we instead model them as a continuum like
in kinetic (Boltzmann, Vlasov) theory \cite{Glassey, Sone, Spohn} but
without any self-interaction. Our focus is on the interaction of the
particles with the body at its boundary. In typical physical scenarios this
interaction can be quite complicated. For instance, the boundary may be so
rough that a particle may reflect from it in an essentially random way.
There could even be some kind of physical or chemical reaction between the
particle and the molecules of the body.

The present paper is a sequel to \cite{CS-1} and is also highly motivated by
the series of papers \cite{Ital1, Ital2, Ital3}. In all four papers the
initial velocity $V(0)$ of the body is close to its terminal (equilibrium)
velocity $V_{\infty }>0$. In \cite{Ital1, Ital3, CS-1} the body approaches
its equilibrium velocity in such a way that $V(t)<V_{\infty }$ for all time $%
t$. On the other hand, in the paper \cite{Ital2} the body's initial velocity 
$V(0)$ satisfies $V(0)>V_{\infty }$ and after a certain time its velocity
switches to $V(t)<V_{\infty }$ before approaching its equilibrium $%
V(t)\rightarrow V_{\infty }$ as $t\rightarrow \infty $. In \cite{Ital2} all
the particles reflect elastically (specularly).

The purpose of the present paper is to analyze the effect of inelastic
(diffusive) collisions given by a probability distribution $K$ and to
determine conditions on $K$ so that the velocity of the body reverses or
does not reverse, that is, $V(t)-V_{\infty }$ does or does not change sign.
We discover that there are diffusive collision laws that lead to reversal
and others that lead to irreversal, no matter what $V_{\infty }$ and $V(0)$
are, so long as they are close together. These laws are almost exact
opposites of each other. In particular, the existence of an irreversal case
for $0<V_{\infty }<V(0)$ is in direct contrast to the purely specular
collision case in \cite{Ital2} where only reversal takes place.

Moreover, in the present paper we prove that, regardless of whether the
velocity is reversed or not, the equilibrium is ultimately approached at the
same polynomial rate as in \cite{CS-1}. This rate is $O(t^{-d-p})$ in $d$
spatial dimensions where $p$ could take any value in $\left( 0,2\right] $,
depending on the specific law of reflection given by $K$. Though purely
diffuse collisions with a Gaussian kernel were considered in \cite{Ital3}
and both diffuse and elastic collisions were considered in \cite{CS-1}, in
both of those papers there was no reversal of the velocity. 
Some discussion of the physical motivation of this type of problem can be
found in \cite{CS-1} and the other cited references. Some closely related
investigations are \cite{ATC,Cav,CM2,Sisti,TA}.

To be specific, here we consider the following problem. The body is a
cylinder $\Omega (t)\subset \mathbb{R}^{d}$. We write $\mathbf{x}%
=(x,x_{\perp }),\ x_{\perp }\in \mathbb{R}^{d-1}$. The cylinder is parallel
to the $x$-axis and the body is constrained to move only in the $x$
direction with velocity $V(t)$. There may be a constant horizontal force $%
E\geq 0$ acting on the body, as well as the horizontal force $F(t)$ due to
all the colliding particles at time $t$. Thus 
\begin{equation*}
\frac{dX}{dt}=V(t),\quad \frac{dV}{dt}=E-F(t),
\end{equation*}%
If $E=0$, then the body is at rest in equilibrium ($V_{\infty }=0$), while
if $E\neq 0$, then $V_{\infty }\neq 0$ is given by $F_{0}(V_{\infty })=E$,
where $F_{0}(V)$ is the fictitious force in case no particle collides more
than once (see (\ref{eqn:F_0}) below). In order to avoid confusion in this
paper, we shall take $0\leq V_{\infty }<V(0)$.

We introduce the following notation. The velocity of a particle is $\mathbf{v%
}=(v_{x},v_{\perp })$, where $v_{x}=\mathbf{v}\cdot \mathbf{i}$ is the
horizontal component and $v_{\perp }\in \mathbb{R}^{d-1}$. The particle
distribution, denoted by $f(t,\mathbf{x},\mathbf{v})$, satisfies $\partial
_{t}f+\mathbf{v}\cdot \nabla _{\mathbf{x}}f=0$ in $\Omega ^{c}(t)$. We
assume the initial velocity $f(0,\mathbf{x},\mathbf{v})=f_{0}(\mathbf{v})$
depends only on $\mathbf{v}$ and is even in $v_{x}$. We also denote the
densities before and after a collision with the body by $f_{\pm }(t,\mathbf{x%
},\mathbf{v})=\lim_{\epsilon \rightarrow 0^{+}}f(t\pm \epsilon ,\mathbf{x}%
\pm \epsilon \mathbf{v},\mathbf{v})$. The assumed law of reflection at the
two ends of the cylinder is 
\begin{equation}
f_{+}(t,\mathbf{x};\mathbf{v}) = \int_{(u_{x}-V\left( t\right)
)(v_{x}-V\left( t\right) )\leq 0}K\left( \mathbf{v}-\mathbf{i}V\left(
t\right) ;\mathbf{u}-\mathbf{i}V\left( t\right) \right) f_{-}(t,\mathbf{x};%
\mathbf{u})d\mathbf{u},
\end{equation}%
where $\mathbf{i}$ is the unit vector in the $x$-direction. The collision
kernel $K$ is assumed to be even and satisfy the conservation of mass
condition \eqref{mass conserv} below. Furthermore, $K$ and the initial
density $f_{0}$ are assumed to satisfy Assumptions A1-A4 in Section \ref%
{Sec:AssumptionsOnKernels}. Among these conditions are $f_0(\mathbf{v}) =
a_0(v_x) b(v_\perp)$ and 
\begin{equation*}
K(\mathbf{v,u})=k(v_{x},u_{x})b(v_{\bot }),\quad c\left\vert
u_{x}\right\vert ^{p}\leqslant \ \int_{v_{x}\geq 0}\ v_{x}^{2}\
k(v_{x},u_{x})\ dv_{x}\ \leqslant C\left\vert u_{x}\right\vert ^{p}
\end{equation*}%
for some constants $c,C,p$ and some function $b(v_{\perp })$ where $0<p\leq
2 $.


\begin{theorem}[Irreversal]
\label{ThIRExistence} Let $K$ be a collision kernel as above and let $f_{0}$
be the initial particle distribution. Let the initial velocity $V_{0}$ of
the cylinder be slightly larger than $V_{\infty }$; that is, $0\leqslant
V_{\infty }<V_{0}<V_{\infty }+\gamma $, where $\gamma $ is sufficiently
small. Assume the Irreversal Criterion 
\begin{equation}
\int_{0}^{\infty }k(0,z)\ a_{0}(z+V_{\infty })\ dz>a_{0}(V_{\infty }).
\label{IRREVcrit}
\end{equation}

(a) Then there exists at least one solution $(V(t),f(t,x,v))$ of our problem
in the following sense. $V\in C^{1}(\mathbb{R})$ and $f_{\pm }\in L^{\infty
} $ for $t\in \lbrack 0,\infty ),x\in \partial \Omega (t),v\in \mathbb{R}%
^{3} $, where the force $F(t)$ on the cylinder is given by \eqref{Force}
below 
and the pair of functions $f_{\pm }(t,x,v)$ are (almost everywhere) defined
explicitly in terms of $V(t)$ and $f_{0}(x,v)$.

(b) Furthermore, every solution of the problem (in the sense stated above)
satisfies the estimates 
\begin{equation}
0<\gamma e^{-B_{0}t}+\frac{c\gamma ^{p+1}}{t^{d+p}}\chi \{t\geq
t_{0}+1\}<V(t)-V_{\infty }<\gamma e^{-B_{\infty }t}+\frac{C\gamma ^{p+1}}{%
(1+t)^{d+p}}  \label{IrrEst}
\end{equation}%
for $0<t<\infty $ and for some positive constants $c$, $C$, $B_{0}$, $%
B_{\infty }$ and $t_{0}$ that will be specified later. Notice that there is
no velocity reversal because $V(t)>V_{\infty }$ for all $t>0$.
\end{theorem}

\begin{theorem}[Reversal]
\label{ThRExistence} Given the same situation as above, except that we now
assume the Reversal Criterion 
\begin{equation}  \label{REVcrit}
\int_0^\infty k(0,z) \ a_0(z+V_\infty) \ dz < a_0(V_\infty).
\end{equation}

(a) Then there exists at least one solution $(V(t),f(t,x,v))$ of our problem
in the sense given in part (a) of the preceding theorem.

(b) Furthermore, every solution of the problem (in the sense stated above)
satisfies the estimates 
\begin{equation}
\gamma e^{-B_{0}t}-\frac{c\gamma ^{p+1}}{(1+t)^{d+p}}<V(t)-V_{\infty
}<\gamma e^{-B_{\infty }t}-\frac{C\gamma ^{p+1}}{t^{d+p}}\chi \{t\geqslant
t_{0}+1\}  \label{RevEst}
\end{equation}%
for $0<t<\infty $ and for some positive constants $c$, $C$, $B_{0}$, $%
B_{\infty }$ and $t_{0}$ specified later. Notice that the velocity reversal
for sufficiently large $t$ is incorporated in this inequality.
\end{theorem}

The contrasting criteria \eqref{IRREVcrit} and \eqref{REVcrit} have the
following interpretation. The body is initially moving to the right. Letting 
$u_{x}=z+V_{\infty }$, we see that the left side of both inequalities
represents the velocity density, after collisions on the left of the body,
of the particles with approximately the same velocity as the body. These are
the particles that are most likely to collide again later. The particles on
the right are less likely to collide with the body again because the body is
slowing down initially under the condition that $V_{0}>V_{\infty }$. In the reversal
case there will be fewer collisions on the left side compared with the
particles that do not collide. Therefore there are fewer future collisions
on the left, so that the body tends to move more to the left, and $V(t)$ has
more of a chance to cross over from being larger than $V_{\infty }$ to being
smaller. In the irreversal case, there are more such particles, so there are
more collisions on the left and the velocity of the body tends to remain
larger than $V_{\infty }$. Much more subtle, and not studied in this paper,
is the case when there is equality in \eqref{IRREVcrit} and \eqref{REVcrit}.

In case $0<V_{0}<V_{\infty }$, the body initially moves slower than the
equilibrium, so the particles on the right now play the critical role.
Adapting \eqref{IRREVcrit} and \eqref{REVcrit} to the right side of the
cylinder, the Irreversal Criterion becomes%
\begin{equation*}
\int_{-\infty }^{0}k(0,z)\ a_{0}(z+V_{\infty })\ dz>a_{0}(V_{\infty }),
\end{equation*}%
which was treated in \cite{CS-1}, while the Reversal Criterion becomes%
\begin{equation*}
\int_{-\infty }^{0}k(0,z)\ a_{0}(z+V_{\infty })\ dz<a_{0}(V_{\infty }),
\end{equation*}%
which has a very similar proof that we omit. In case $V_{\infty }=0$, the
criteria on the right and the left coincide due to the evenness of $K$ and $%
f_{0}$.

We now discuss the basic setup of the problem which is essentially the same
as in \cite{CS-1}, where more details and derivations may be found. The
kernel $k(v_{x},u_{x})$ is assumed to be nonnegative and even in each
variable. Conservation of mass requires that 
\begin{equation}
\int_{v_{x}\geqslant 0}v_{x}K(\mathbf{v,u})d\mathbf{v}=\left\vert
u_{x}\right\vert .  \label{mass conserv}
\end{equation}%
The total horizontal force on the body at time $t$ due to the particles is
given in terms of $f_{-}(t,\mathbf{x},\mathbf{v})$ as 
\begin{equation}
F(t)=\int_{\partial \Omega _{R}(t)\cup \partial \Omega _{L}(t)}dS_{\mathbf{x}%
}\int_{\mathbb{R}^{3}}d\mathbf{v}\ \text{sgn}(V(t)-v_{x})\ \ell (\mathbf{v}-%
\mathbf{i}V(t))\ f_{-}(t,\mathbf{x},\mathbf{v}),  \label{Force}
\end{equation}%
where we denote 
\begin{equation}
\ell (\mathbf{w})=w_{x}^{2}+\int_{v_{x}\geq 0}d\mathbf{v}\ v_{x}^{2}\ K(%
\mathbf{v},\mathbf{w}).  \label{ell}
\end{equation}%
On the lateral boundary $S$ of the cylinder we also assume a boundary
condition of the form 
\begin{equation*}
f_{+}(t,\mathbf{x};\mathbf{v})=\int_{\mathbf{u}\cdot \mathbf{n}_{\mathbf{x}%
}\leqslant 0}K_{S}\left( \mathbf{v};\mathbf{u}\right) f_{0}(\mathbf{u})d%
\mathbf{u.}
\end{equation*}%
together with the corresponding conservation of mass condition. Then no net
force is created on the lateral boundary (see \cite[Lemma 2.5]{CS-1}) and
the body continues to move horizontally.

In Section 2 we state the precise assumptions on the collision kernel $K$
and the initial particle distribution $f_{0}$, followed by several examples.
Example 1 is a Gaussian collision law of both $k$ and $a_{0}$, namely 
\begin{equation*}
a_{0}(u_{x}) = C_1e^{-\alpha u_{x}^{2}}, \quad k(v_{x},u_{x}) = C_2e^{-\beta
v_{x}^{2}}\left\vert u_{x}\right\vert .
\end{equation*}
If $V_\infty$ either vanishes or is small enough, then it satisfies the
Reversal Criterion if $\beta<\alpha$, while it satisfies the Irreversal
Criterion provided $\beta>\alpha$. If $V_\infty^2 > \frac\beta{2\alpha^2}$,
it satisfies the Reversal Criterion. Example 2 has a Gaussian kernel $K(v,u) 
$ like $\exp (-v_{x}^{2}/|u_{x}|)$, which means that colliding particles
with velocities close to that of the body deflect only a little, while
colliding particles with very different velocities reflect with a very wide
distribution of velocities. Example 3 is more general, permitting an
ultimate rate of approach to the body at the rate $O(t^{-d-p})$ for any $%
p\in (0,2]$.

Sections 3 and 4 are devoted to the proofs of the irreversal and reversal
cases, respectively. In each case a family $\mathcal{W}$ of possible body
motions $W$ is introduced. We write the force due to the possible motion $W$
as $F(t)=F_{0}(t)+R_{W}(t)$, where $R_{W}(t)$ is the force due to the
collisions occurring before time $t$ (\textquotedblleft precollisions") if
the body were to move with velocity $W(\cdot )$. Then $W$ generates a new
possible motion $V_{W}$ by the equation 
\begin{equation*}
\frac{dV_{W}}{dt}=\frac{F_{0}(V_{\infty })-F_{0}(W(t))}{V_{\infty }-W(t)}%
\left( V_{\infty }-V_{W}\right) -R_{W}\left( t\right) 
\end{equation*}%
The goal is to prove that the mapping $W\rightarrow V_{W}$ has a fixed
point. The main upper and lower bounds of $R_{W}(t)$ are stated in Theorem %
\ref{Thm:IREstimate} for the irreversal case and Theorem \ref{Thm:REstimate}
for the reversal case. Assuming them, Theorems \ref{ThIRExistence} and \ref%
{ThRExistence} follow easily. The proofs of Theorems \ref{Thm:IREstimate}
and \ref{Thm:REstimate} are the core of this paper. The key bound (\ref%
{IR:RestimateL}) in the irreversal case proves that $R_{W}(t)$ remains
negative. On the other hand, $R_{W}(t)$ stays positive in the reversal case,
according to (\ref{R:RestimateL}). We begin the proofs by considering the
class $\mathcal{W}$ of possible motions and then prove the required bounds
for the particles colliding with the body from the left side, followed by
those that collide from the right side. In the reversal case, because $E\geq
0$, the collisions on the right begin to dominate but, after the velocity
reversal, eventually those on the right and the left balance each other and
the body tends to its equilibrium speed from below.

\subsection{Acknowledgement}

We thank Carlo Marchioro for his helpful comments on the related paper \cite%
{Ital2} with purely elastic collisions. This research was supported in part
by NSF grant DMS-1007960.


\section{Assumptions and Examples of Diffusion Kernels\label%
{Sec:AssumptionsOnKernels}}

We make the following assumptions on the diffusion kernel $K$, which governs
the collisions with the body.

\begin{itemize}
\item[A1.] [Structure] Let $K$ and $f_{0}$ have the product form 
\begin{eqnarray*}
f_{0}(\mathbf{v}) &=&a_{0}(v_{x})b(v_{\bot }), \\
K(\mathbf{v,u}) &=&k(v_{x},u_{x})b(v_{\bot }),\qquad \int b(v_{\bot
})dv_{\bot }=1\text{ and }b(0) > 0\text{,}
\end{eqnarray*}%
with each factor nonnegative and continuous, with $f_{0}$ bounded and with $%
a_{0}$ and $k$ even functions in both $u_{x}$ and $v_{x}$.

Notice that, under this assumption, $\ell (\mathbf{w})$ actually depends
only on $w_{x}$; that is, 
\begin{equation*}
\ell (\mathbf{w}) = w_x^2 + \int_{v_{x}\geq 0}dv_{x}\ v_{x}^{2}\
k(v_{x},w_{x})=\ell (w_{x}).
\end{equation*}%
Therefore, at any later time $f_{+}$ and $f_{-}$ must take the product form 
\begin{equation*}
f_{+}(t,\mathbf{x};\mathbf{v})=a_{+}(t,\mathbf{x;}v_{x})b(v_{\bot }),\quad
f_{-}(t,\mathbf{x};\mathbf{v})=a_{-}(t,\mathbf{x;}v_{x})b(v_{\bot }).
\end{equation*}%
We remark that, although collisions occur only in the horizontal direction,
the analysis is not entirely one-dimensional; the dimension does come into
play as we shall see in the proofs of \eqref{IrrEst} and \eqref{RevEst}.
\bigskip 

\item[A2.] [Boundedness] \noindent 
\begin{equation*}
\sup_{\left\vert u_{x}\right\vert \leqslant \gamma } \ \sup_{v_{x}\in 
\mathbb{R}} \ k(v_{x},u_{x})<\infty .
\end{equation*}

\item[A3.] [Power Law] There is a power $0<p\leq 2$ and there are positive
constants $C$ and $c$ such that 
\begin{equation*}
c\left\vert u_{x}\right\vert ^{p}\leqslant \int_{v_{x}\geq 0}\ v_{x}^{2}\
k(v_{x},u_{x})\ dv_{x}\ \leqslant C\left\vert u_{x}\right\vert ^{p}
\end{equation*}%
for $u_{x}\in \left[ -\gamma ,\gamma \right] $. We also assume that this
integral is an even $C^{1}$ function of $u_{x}$ for $u_{x}\neq 0$ and is
strictly decreasing for $u_{x}<0$. Combining A3 and A1, we have 
\begin{equation*}
c|u_{x}|^{p}\leqslant \ell (u_{x})\leqslant u_{x}^{2}+C|u_{x}|^{p}\leqslant
C^{\prime }|u_{x}|^{p}\text{ for }u_{x}\in \left[ -\gamma ,\gamma \right] .
\end{equation*}

\item[A4.] [Integrability] 
\begin{equation*}
k(v_{x},z-y-V_{\infty })\ a_{0}(z)\leq M(z)\quad \text{ for }|v_{x}|<2\gamma
,\ |y|<\gamma ,\ |z|<\infty ,
\end{equation*}%
where $M\in L^{1}(\mathbb{R})$.
\end{itemize}


\subsection{Examples of Collision Kernels}

In this section, we give a few examples of collision kernels and initial
densities that satisfy the assumptions.

\begin{example}
Let 
\begin{equation*}
a_{0}(u_{x})=C_{1}e^{-\alpha u_{x}^{2}},\quad k(v_{x},u_{x})=C_{2}e^{-\beta
v_{x}^{2}}\left\vert u_{x}\right\vert .
\end{equation*}%
The requirement that mass is conserved means that 
\begin{equation*}
\int_{v_{x}\geqslant 0}v_{x}K(\mathbf{v,u})d\mathbf{v}=\left\vert
u_{x}\right\vert ,
\end{equation*}%
which reduces to choosing $C_{2}=2\beta $. Assumptions A1-A4 are seen to be
easily satisfied with $p=1$. The Reversal Criterion takes the form 
\begin{equation}
2\beta \int_{0}^{\infty }ze^{-\alpha (z+V_{\infty })^{2}}dz=2\beta
\int_{V_{\infty }}^{\infty }(z-V_{\infty })e^{-\alpha z^{2}}dz<e^{-\alpha
V_{\infty }^{2}},  \label{RCexample1}
\end{equation}%
while the Irreversal Criterion is the opposite (strict) inequality.

Let us first suppose that $V_{\infty }$ either vanishes or is very small.
Then the Reversal Criterion is satisfied provided $\frac{\beta }{\alpha }%
=2\beta \int_{0}^{\infty }ze^{-\alpha z^{2}}dz<1$, or $\beta <\alpha $,
while the Irreversal Criterion is satisfied if $\beta >\alpha $. Now $\alpha 
$ and $\beta $ may be interpreted as the reciprocals of (normalized)
temperatures. So the velocity reverses if the body is hotter than the gas
and the speed $V_{\infty }$ is sufficiently small. The velocity does not
reverse if the body has a lower temperature than the gas and the speed $%
V_{\infty }$ is sufficiently small. The latter situation could happen for a
comet or a space vehicle that actively cools itself during its reentry into
the atmosphere.

Next let us consider a fast moving body. We can write \eqref{RCexample1} as 
\begin{equation*}
\frac{1}{2\alpha }e^{-\alpha V_{\infty }^{2}}-\frac{V_{\infty }\sqrt{\pi }}{2%
\sqrt{\alpha }}{erfc}(\sqrt{\alpha }V_{\infty })<\frac{1}{2\beta }e^{-\alpha
V_{\infty }^{2}}.
\end{equation*}%
We can use the asymptotic expansion of the complementary error function 
\begin{equation*}
{erfc}\left( x\right) \sim \frac{e^{-x^{2}}}{x\sqrt{\pi }}\sum_{n=0}^{\infty
}(-1)^{n}\frac{\left( 2n-1\right) !!}{\left( 2x^{2}\right) ^{n}}.
\end{equation*}%
In fact, two easy integrations by parts yield 
\begin{equation*}
{erfc}(x)=\frac{e^{-x^{2}}}{x\sqrt{\pi }}\left( 1-\frac{1}{2x^{2}}\right) +%
\frac{3}{2\sqrt{\pi }}\int_{x}^{\infty }e^{-t^{2}}\frac{dt}{t^{4}}\ >\ \frac{%
e^{-x^{2}}}{x\sqrt{\pi }}\left( 1-\frac{1}{2x^{2}}\right)
\end{equation*}%
for any $x>0$. Thus the Reversal Criterion is satisfied if 
\begin{equation*}
\frac{1}{4\alpha ^{2}V_{\infty }^{2}}e^{-\alpha V_{\infty }^{2}}=\frac{1}{%
2\alpha }e^{-\alpha V_{\infty }^{2}}-\frac{1}{2\alpha }\left( 1-\frac{1}{%
2\alpha V_{\infty }^{2}}\right) e^{-\alpha V_{\infty }^{2}}<\frac{1}{2\beta }%
e^{-\alpha V_{\infty }^{2}}.
\end{equation*}%
That is, the velocity reverses if $V_{\infty }^{2}>\frac{\beta }{2\alpha ^{2}%
}$. In particular, if $\alpha =\beta $, the velocity reverses if $V_{\infty
}^{2}>\frac{1}{2\alpha }$. This agrees with the numerical calculations of
Case 9 in \cite{ATC}. The case of equal temperatures ($\alpha =\beta $) is
motivated by Boltzmann theory \cite{Sone, Spohn}.

Now consider a fast moving body and look for irreversal. We further expand 
\begin{equation*}
{erfc}(x)=\frac{e^{-x^{2}}}{x\sqrt{\pi }}\left( 1-\frac{1}{2x^{2}}+\frac{3}{%
4x^{4}}\right) -\frac{15}{4\sqrt{\pi }}\int_{x}^{\infty }e^{-t^{2}}\frac{dt}{%
t^{6}}.
\end{equation*}%
Dropping the last integral, we deduce that there is no reversal if $\frac{1}{%
4\alpha ^{2}V_{\infty }^{2}}-\frac{3}{8\alpha ^{3}V_{\infty }^{4}}>\frac{1}{%
2\beta }$, which can happen if $\beta >12\alpha $. This again agrees with
the numerical data in \cite{ATC}, namely that there is reversal if $\alpha
=\beta $.
\end{example}


\begin{example}
\label{example:p=3/2} We now choose 
\begin{equation*}
K(\mathbf{v,u})=2e^{-\frac{v_{x}^{2}}{\left\vert u_{x}\right\vert }%
}b(v_{\bot }),\text{ }f_{0}(\mathbf{v})=a_{0}(v_{x})b(v_{\bot }),
\end{equation*}%
as in \cite[Example 2]{CS-1}. Mass conservation during collisions forces the
coefficient to be $2$. We assume $a_{0}\in L^{1}(\mathbb{R})$, and $\int
bdv_{\bot }=1.$ The physical interpretation is that an incoming particle
with almost the same velocity as that of the body is likely to be reflected
with almost the same velocity, while an incoming particle with velocity
quite different from that of the body is reflected according to a wide
Gaussian distribution around $V(t).$ As in \cite[Example 2]{CS-1}, this
collision kernel satisfies A1-A4 with $p=\frac{3}{2}$. The Reversal
Criterion then means that 
\begin{equation}
\int_{V_{\infty }}^{\infty }a_{0}(u)du<\frac{a_{0}(V_{\infty })}{2},
\label{verifyA5forEx2-1}
\end{equation}%
while the Irreversal Criterion means the opposite (strict) inequality. A
simple instance of reversal is the algebraic decay: $a_{0}(u)=\frac{1}{u^{m}}
$ for $u\geqslant 1$, in which case the Reversal Criterion is satisfied so
long as $1\leqslant V_{\infty }<\frac{m-1}{2}$ with $m>4$. A second instance
is the Gaussian $a_{0}(u)=C_{1}e^{-\beta u^{2}}$, for which reversal means 
\begin{equation*}
\frac{2}{\sqrt{\beta }}\int_{\sqrt{\beta }V_{\infty }}^{\infty
}e^{-z^{2}}dz<e^{-\beta V_{\infty }^{2}}.
\end{equation*}%
Using one term in the expansion of erfc with a negative remainder, we see
that reversal occurs if $V_{\infty }<\frac{1}{\beta }$. Similarly, using two
terms in the expansion with a positive remainder, we see that irreversal
occurs if $\frac{1}{\beta V_{\infty }}-\frac{1}{2\beta ^{2}V_{\infty }^{3}}%
>1 $.
\end{example}



\begin{example}
\label{example:range of p} As in \cite[Example 3]{CS-1}, we can find a
family of kernels that covers a continuous range of $p.$ Given $\beta \in %
\left[ -1,3\right) $, we choose 
\begin{equation*}
K(\mathbf{v,u})=C_{2}\left\vert u_{x}\right\vert ^{\beta }e^{-{v_{x}^{2}}{%
\left\vert u_{x}\right\vert ^{\beta -1}}}b(v_{\bot }),\quad f_{0}(\mathbf{v}%
)=a_{0}(v_{x})b(v_{\bot }).
\end{equation*}%
Once again, $C_{2}$ is chosen so that mass is conserved during collisions,
while $a_{0}$ and $b$ are chosen as in Example \ref{example:p=3/2}. We then
have%
\begin{equation*}
C_{2}\left\vert u_{x}\right\vert ^{\beta }\int_{0}^{\infty }v_{x}^{2}\ e^{-{%
v_{x}^{2}}{\left\vert u_{x}\right\vert ^{\beta -1}}}dv_{x}=C\left\vert
u_{x}\right\vert ^{\frac{3-\beta }{2}}
\end{equation*}%
for some constant $C$. Thus $p$ runs through $\left( 0,{2}\right] $ as $%
\beta $ runs through $\left[ -1,3\right) .$
\end{example}

\section{Proof of the Irreversal Case\label{Sec:IR}}

\subsection{Proof assuming the Key Estimate}

Theorem \ref{ThIRExistence} will be proven by a fixed point technique. We
will first define a family $\mathcal{W}$ of possible body motions $W$. Given
a possible motion $W\in \mathcal{W},$ let $F_{0}(W)$ be the force if each
particle were to collide only once and the body were to move with velocity $%
W(t)$. It is given by putting $f_{0}$ in place of $f_{-}$ in (\ref{Force}),
that is, 
\begin{eqnarray}
F_{0}(W) &=&\int_{\partial \Omega _{R}(t)}dS_{\mathbf{x}}\int_{u_{x}%
\leqslant W(t)}d\mathbf{u}\ \ell (\mathbf{u}-\mathbf{i}W(t))f_{0}(\mathbf{u})
\label{eqn:F_0} \\
&&-\int_{\partial \Omega _{L}(t)}dS_{\mathbf{x}}\int_{u_{x}\geqslant W(t)}d%
\mathbf{u}\ \ell (\mathbf{u}-\mathbf{i}W(t))f_{0}(\mathbf{u})  \notag \\
&=&C\left( \int_{u_{x}\leqslant W(t)}\ell (\mathbf{u}-\mathbf{i}W)f_{0}(%
\mathbf{u})d\mathbf{u}-\int_{u_{x}\geqslant W(t)}\ell (\mathbf{u}-\mathbf{i}%
W)f_{0}(\mathbf{u})d\mathbf{u}\right) ,  \notag
\end{eqnarray}%
where $C=\left\vert \partial \Omega _{L}\right\vert $. By \cite[Lemma 2.8]%
{CS-1}, $F_{0}(W)$ is a positive, increasing $C^{1}$ function of $W.$ Let $%
R_{W}(t)=F(t)-F_{0}(W(t))$ be the force due to the collisions with the
particles that occurred before time $t$ (\textquotedblleft precollisions")
if the body were to move with velocity $W(t)$, given by the formula%
\begin{eqnarray*}
R_{W}(t) &=&\int_{\partial \Omega _{L}(t)}dS_{\mathbf{x}}\int_{u_{x}%
\geqslant W(t)}d\mathbf{u}\ \ell (\mathbf{u}-\mathbf{i}W(t))\left\{ f_{0}(%
\mathbf{u})-f_{-}(t,\mathbf{x},\mathbf{u})\right\} \\
&&+\int_{\partial \Omega _{R}(t)}dS_{\mathbf{x}}\int_{u_{x}\leqslant W(t)}d%
\mathbf{u}\ \ell (\mathbf{u}-\mathbf{i}W(t))\left\{ f_{-}(t,\mathbf{x},%
\mathbf{u})-f_{0}(\mathbf{u})\right\} .
\end{eqnarray*}%
As mentioned in the introduction, we then define a new motion $V_{W}$ by
means of the equation 
\begin{equation}
\frac{dV_{W}}{dt}=\frac{F_{0}(V_{\infty })-F_{0}(W(t))}{V_{\infty }-W(t)}%
\left( V_{\infty }-V_{W}\right) -R_{W}\left( t\right) .
\label{iteration equation}
\end{equation}%
The main part of the proof is to establish, as stated in Theorem \ref%
{Thm:IREstimate} below, an upper and a lower bound of $R_{W}(t)$ for all $%
W\in \mathcal{W}.$ Using Theorem \ref{Thm:IREstimate}, we will prove by
means of Lemma \ref{IR:DeducingConditionsOnhAndg} that $V_{W}\in \mathcal{W}$%
.

\begin{definition}[Class of possible motions for the irreversal case]
\label{def:W-IR}We define $\mathcal{W}$ as the family of functions $W$ that
satisfy the following conditions.

(i) $W:[0,\infty )\rightarrow \mathbb{R}$ is Lipschitz and $W(0)=V_{\infty
}+\gamma $.

(ii) $W$ is decreasing over the interval $[0,t_{0}]$ for $t_{0}=\left\vert
\ln \gamma \right\vert $.

(iii) For all $W\in \mathcal{W}$, $t\in \lbrack 0,\infty )$ and $\gamma \in
(0,1)$, 
\begin{equation}
\gamma h(t,\gamma )\leqslant V_{\infty }-W(t)\leqslant \gamma g(t,\gamma ),
\label{g and h}
\end{equation}%
that is, 
\begin{equation*}
V_{\infty }-\gamma g(t,\gamma )\leqslant W(t)\leqslant V_{\infty }-\gamma
h(t,\gamma ),
\end{equation*}%
where%
\begin{eqnarray*}
-g(t,\gamma ) &=&e^{-B_{0}t}+\frac{\gamma ^{p}A_{+}}{t^{p+d}}\chi \left\{
t\geqslant t_{0}+1\right\} , \\
-h(t,\gamma ) &=&e^{-B_{\infty }t}+\frac{\gamma ^{p}A_{-}}{\left\langle
t\right\rangle ^{p+d}}.
\end{eqnarray*}%
with%
\begin{equation*}
B_{0}=\max_{V\in \left[ V_{\infty }-\gamma ,V_{\infty }+\gamma \right]
}F_{0}^{\prime }(V),B_{\infty }=\min_{V\in \left[ V_{\infty }-\gamma
,V_{\infty }+\gamma \right] }F_{0}^{\prime }(V).
\end{equation*}
\end{definition}

\begin{theorem}
\label{Thm:IREstimate}If $k$ and $a_{0}$ satisfy the Irreversal Criterion,
then for all small enough $\gamma $, there exists $c_{1}$, $C_{2},$ and $C>0$
such that for all $W\in \mathcal{W}$, we have%
\begin{equation}
R_{W}(t)\leqslant \left[ -c_{1}\frac{\gamma ^{p+1}}{t^{p+d}}+\frac{%
C_{2}\gamma ^{p+1}A_{-}^{p+1}}{\left\langle t\right\rangle ^{\left(
p+d\right) \left( p+1\right) }}\right] \chi \left\{ t\geqslant t_{0}\right\}
\leqslant 0  \label{IR:RestimateL}
\end{equation}
and 
\begin{equation}
R_{W}(t)\geqslant -\frac{C\left( \gamma +\gamma ^{p+1}A_{-}\right) }{%
\left\langle t\right\rangle ^{p+d}}^{p+1} .  \label{IR:RestimateU}
\end{equation}
\end{theorem}

\begin{proof}
We postpone the proof of Theorem \ref{Thm:IREstimate} to the end of Section
3.
\end{proof}


\begin{lemma}
\label{IR:DeducingConditionsOnhAndg}If $k$ and $a_{0}$ satisfy the
Irreversal Criterion, for small enough $\gamma $, we can choose $A_{+}$ and $%
A_{-}$ in Definition \ref{def:W-IR} such that, for any $W\in \mathcal{W}$,
the solution $V_{W}$ to the iteration equation (\ref{iteration equation}) 
\begin{equation*}
\frac{dV_{W}}{dt}=Q(t)\left( V_{\infty }-V_{W}\right) -R_{W}\left( t\right)
,\quad Q(t)=\frac{F_{0}(V_{\infty })-F_{0}(W(t))}{V_{\infty }-W(t)},
\end{equation*}%
satisfies 
\begin{equation*}
-\gamma e^{-B_{\infty }t}-\frac{A_{-}\gamma ^{p+1}}{\left( 1+t\right) ^{p+d}}%
\leqslant V_{\infty }-V_{W}(t)\leqslant -\gamma e^{-B_{0}t}-\frac{%
A_{+}\gamma ^{p+1}}{t^{p+d}}\chi \left\{ t\geqslant t_{0}+1\right\} <0.
\end{equation*}%
In other words, for every $W\in \mathcal{W},$ we have $V_{W}\in \mathcal{W}.$
\end{lemma}

\begin{proof}
By \eqref{iteration equation} we have%
\begin{equation*}
\frac{d\left( V_{\infty }-V_{W}\right) }{dt}=-Q(t)\left( V_{\infty
}-V_{W}\right) +R_{W}\left( t\right) ,
\end{equation*}%
hence%
\begin{equation*}
V_{\infty }-V_{W}\left( t\right) =-\gamma
e^{-\int_{0}^{t}Q(r)dr}+\int_{0}^{t}\left( e^{-\int_{s}^{t}Q(r)dr}\right)
R_{W}(s)ds
\end{equation*}%
because $V_{\infty }-V_{W}\left( 0\right) =-\gamma $. On the one hand, by (%
\ref{IR:RestimateU}), 
\begin{eqnarray*}
V_{\infty }-V_{W}\left( t\right) &\geqslant &-\gamma e^{-B_{\infty
}t}-\int_{0}^{t}\left( e^{-\int_{s}^{t}Q(r)dr}\right) \frac{C\left( \gamma
+\gamma ^{p+1}A_{-}\right) ^{p+1}}{\left\langle s\right\rangle ^{p+d}}ds \\
&\geqslant &-\gamma e^{-B_{\infty }t}-\int_{0}^{t}e^{-B_{\infty }(t-s)}\frac{%
C\left( \gamma +\gamma ^{p+1}A_{-}\right) ^{p+1}}{\left\langle
s\right\rangle ^{p+d}}ds \\
&=&-\gamma e^{-B_{\infty }t}-C\left( \gamma +\gamma ^{p+1}A_{-}\right)
^{p+1}\int_{0}^{t}\frac{e^{-B_{\infty }(t-s)}}{\left\langle s\right\rangle
^{p+d}}ds
\end{eqnarray*}%
where%
\begin{eqnarray*}
&&\int_{0}^{t}\frac{e^{-B_{\infty }(t-s)}}{\left\langle s\right\rangle ^{p+d}%
}ds \\
&=&\int_{0}^{\frac{t}{2}}e^{-B_{\infty }(t-s)}\frac{1}{\left( 1+s\right)
^{p+d}}ds+\int_{\frac{t}{2}}^{t}e^{-B_{\infty }(t-s)}\frac{1}{\left(
1+s\right) ^{p+d}}ds \\
&\leqslant &\int_{0}^{\frac{t}{2}}e^{-B_{\infty }(t-s)}ds+\frac{C}{\left(
1+t\right) ^{p+d}}\int_{\frac{t}{2}}^{t}e^{-B_{\infty }(t-s)}ds \\
&\leqslant &\frac{1}{B_{\infty }}(e^{-\frac{B_{\infty }t}{2}}-e^{-B_{\infty
}t})+\frac{C}{\left( 1+t\right) ^{p+d}}\ \ \leqslant \ \frac{C^{\prime }}{%
\left( 1+t\right) ^{p+d}},
\end{eqnarray*}%
That is, 
\begin{equation*}
V_{\infty }-V_{W}\left( t\right) \geqslant -\gamma e^{-B_{\infty }t}-\frac{%
C^{\prime }\left( \gamma +\gamma ^{p+1}A_{-}\right) ^{p+1}}{\left(
1+t\right) ^{p+d}}.
\end{equation*}%
Letting $A_{-}>C^{\prime }$, we have $C^{\prime }\left( 1+\gamma
^{p}A_{-}\right) ^{p+1}\leqslant A_{-}$ for small $\gamma $, whence 
\begin{equation*}
V_{\infty }-V_{W}\left( t\right) \geqslant -\gamma e^{-B_{\infty }t}-\frac{%
A_{-}\gamma ^{p+1}}{\left( 1+t\right) ^{p+d}}.
\end{equation*}%
On the other hand, with the aforementioned $A_{-}$, by (\ref{IR:RestimateL}%
), we have for small $\gamma $%
\begin{equation*}
R_{W}(t)\leqslant -c\frac{\gamma ^{p+1}}{t^{p+d}}\chi \left\{ t\geqslant
t_{0}\right\} \leqslant 0.
\end{equation*}%
Thus 
\begin{eqnarray*}
V_{\infty }-V_{W}\left( t\right) &\leqslant &-\gamma e^{-B_{0}t}-c\gamma
^{p+1}\int_{0}^{t}\left( e^{-\int_{s}^{t}Q(r)dr}\right) \frac{1}{s^{p+d}}%
\chi \left\{ s\geqslant t_{0}\right\} ds \\
&\leqslant &-\gamma e^{-B_{0}t}-c\gamma ^{p+1}\int_{t_{0}}^{t}e^{-B_{0}(t-s)}%
\frac{1}{s^{p+d}}ds,
\end{eqnarray*}%
as long as $1+t_{0}<t$, where%
\begin{equation*}
\int_{t_{0}}^{t}e^{-B_{0}(t-s)}\frac{1}{s^{p+d}}ds\geqslant
\int_{t-1}^{t}e^{-B_{0}(t-s)}\frac{1}{s^{p+d}}ds\geqslant e^{-B_{0}}\frac{1}{%
t^{p+d}}.
\end{equation*}%
Hence%
\begin{equation*}
V_{\infty }-V_{W}\left( t\right) \leqslant -\gamma e^{-B_{0}t}-\frac{c\gamma
^{p+1}}{t^{p+d}}\chi \left\{ t\geqslant t_{0}+1\right\} .
\end{equation*}%
Therefore, selecting $A_{+}\leqslant c$ yields%
\begin{equation*}
V_{\infty }-V_{W}\left( t\right) \leqslant -\gamma e^{-B_{0}t}-\frac{%
A_{+}\gamma ^{p+1}}{t^{p+d}}\chi \left\{ t\geqslant t_{0}+1\right\} .
\end{equation*}
\end{proof}

\begin{proof}[\textit{Proof of Theorem \protect\ref{ThIRExistence} (a)}]
The proof is almost identical to that in \cite{Ital2}, so we merely sketch
it here. Let $L=\max \{V_{\infty }+1,E+F_{0}(V_{\infty }+1)+C\gamma ^{p+1}\}$
and $\mathcal{K}=\{W\in \mathcal{W}\ |\ \sup (|W(t)|+|W^{\prime }(t)|\leq
L\} $. Then $\mathcal{K}$ is a compact convex set in $C([0,\infty ))$. We
define an operator $\mathcal{A}:W\rightarrow V_{W}$, where $V_{W}$ is
defined in \eqref{iteration equation}. Then $\mathcal{A}$ maps $\mathcal{K}$
into itself by Lemma \ref{IR:DeducingConditionsOnhAndg}. By the Schauder
fixed point theorem it suffices to prove that $\mathcal{A}$ is continuous in
the topology of $C([0,\infty ))$.

In order to accomplish that task, we let $W_{j}\rightarrow W$ in $%
C([0,\infty ))$, where $W_{j}\in \mathcal{K}$. Fix $T>0$ so large that the
interval $(T,\infty )$ provides a negligible contribution due to the uniform
decay in Theorem \ref{Thm:IREstimate}. Let $N$ be a positive integer. Define 
$A_{j}^{N}$ be the set of all pairs $(x,v_{x})$ such that no trajectory
passing through $(T,x,v_{x})$ has collided more than $N$ times in the time
interval $[0,T]$. Let $B_{j}^{N}$ be its complement. We write $%
R_{W_{j}}(t)=R_{W_{j}}(t;A_{j}^{N})+R_{W_{j}}(t;B_{j}^{N})$. By Theorem \ref%
{Thm:IREstimate} we have $|R_{W_{j}}(t;B_{j}^{N})|\leq (C\gamma ^{p+1})^{N}$
for $t\leq T$, while $R_{W_{j}}(t;A_{j}^{N})\rightarrow R_{W}(t;A^{N})$ in $%
C([0,T])$ by $N$ uses of the boundary condition. It follows that $%
R_{W_{j}}(t)\rightarrow R_{W}(t)$ uniformly in $[0,T].$
\end{proof}

\begin{proof}[\textit{Proof of Theorem \protect\ref{ThIRExistence} (b)}]
If $(V,f)$ is a solution in the sense of Theorem \textit{\ref{ThIRExistence}}%
, then it is a fixed point of $\mathcal{A}$, so that Theorem \ref%
{Thm:IREstimate} is valid for it. We need only check that the strict
inequalities 
\begin{equation}
\gamma e^{-B_{0}t}+\frac{A_{+}\gamma ^{p+1}}{t^{p+d}}\chi \left\{ t\geqslant
t_{0}+1\right\} <V(t)-V_{\infty }<\gamma e^{-B_{\infty }t}+\frac{A_{-}\gamma
^{p+1}}{\left( 1+t\right) ^{p+d}},  \label{asymp4}
\end{equation}%
are valid for small $t>0$. Indeed, note that at $t=0$ we have $\gamma
=V(0)-V_{\infty }<\gamma +C\gamma ^{p+1}$ and $V^{\prime
}(0)=E-F(0)=F_{0}(V_{\infty })-F_{0}(V_{\infty }+\gamma )<-\gamma \min
(F_{0}^{\prime })\leq -\gamma B_{\infty }<0$. Therefore \eqref{asymp4} is
valid for very small $t>0$ and hence for all $t>0$.
\end{proof}


\subsection{Properties of $\mathcal{W\label{sec:IR:properties of family}}$}


In this subsection we begin the proof of Theorem \ref{Thm:IREstimate}. We
split $R_{W}(t)=r_{W}^{L}\left( t\right) +r_{W}^{R}\left( t\right) $, where
the left contribution $r_{W}^{L}\left( t\right) $ and the right contribution 
$r_{W}^{R}\left( t\right) $ are given by 
\begin{eqnarray*}
r_{W}^{L}\left( t\right) &=&\int_{\partial \Omega _{L}\left( t\right) }dS_{%
\mathbf{x}}\int_{u_{x}\geqslant W\left( t\right) }d\mathbf{u}\ \ell (\mathbf{%
u}-\mathbf{i}W(t))\left\{ f_{0}(\mathbf{u})-f_{-}(t,\mathbf{x},\mathbf{u}%
)\right\} , \\
r_{W}^{R}\left( t\right) &=&\int_{\partial \Omega _{R}\left( t\right) }dS_{%
\mathbf{x}}\int_{u_{x}\leqslant W\left( t\right) }d\mathbf{u}\ \ell (\mathbf{%
u}-\mathbf{i}W(t))\left\{ f_{-}(t,\mathbf{x},\mathbf{u})-f_{0}(\mathbf{u}%
)\right\} .
\end{eqnarray*}%
We shall first prove some properties of the iteration family $\mathcal{W}$
defined in Definition \ref{def:W-IR}. Then we shall estimate $%
r_{W}^{L}\left( t\right) $ and $r_{W}^{R}\left( t\right) $ in Lemmas \ref%
{Lemma:IRLeft} and \ref{Lemma:IRRight}, from which Theorem \ref%
{Thm:IREstimate} will follow.


For any function $Y:[0,\infty )\rightarrow \mathbb{R}$, we denote its
average over time intervals by%
\begin{equation*}
\left\langle Y\right\rangle _{s,t}=\frac{1}{t-s}\int_{s}^{t}Y\left( \tau
\right) d\tau ,\qquad \left\langle Y\right\rangle _{0,t}=\left\langle
Y\right\rangle _{t}.
\end{equation*}%
Thus $Y\in L^{1}(\mathbb{R})$ implies $\langle Y\rangle _{t}=O(1/t)$ for
large $t$. The family $\mathcal{W}=\left\{ W\right\} $, defined in Definiton %
\ref{def:W-IR} for the irreversal case, has the following properties.

\begin{lemma}
\label{Lemma:IR<W>t-W(t)}Let $\mathcal{W}$ be defined in Definiton \ref%
{def:W-IR} for the irreversal case. For all small enough $\gamma $ and hence
for all large enough $t_{0}$, we have%
\begin{equation*}
\left\langle W\right\rangle _{t}-W(t)\geqslant \frac{C\gamma }{t}\text{ for }%
t\geqslant t_{0},
\end{equation*}%
and%
\begin{equation*}
\left\langle W\right\rangle _{t}-W(t)\leqslant \frac{C\left( \gamma +\gamma
^{p+1}A_{-}\right) }{1+t}\text{, for all }t\geq 0.
\end{equation*}
\end{lemma}

\begin{proof}
On the one hand,%
\begin{eqnarray*}
\left\langle W\right\rangle _{t}-W(t) &=&\frac{1}{t}\int_{0}^{t}W(s)ds-W(t)
\\
&\geqslant &\frac{1}{t}\int_{0}^{t}\left( V_{\infty }+\gamma
e^{-B_{0}s}\right) ds -\left( V_{\infty }+\gamma e^{-B_{\infty }t}+\frac{%
\gamma ^{p+1}A_{-}}{\left\langle t\right\rangle ^{p+d}}\right) \\
&\geqslant &\frac{C_{1}\gamma }{t}-\gamma e^{-B_{\infty }t}-\frac{\gamma
^{p+1}A_{-}}{\left\langle t\right\rangle ^{p+d}}
\end{eqnarray*}%
Notice that, for all small enough $\gamma $ and hence all large enough $%
t_{0} $, the second and the third terms are absorbed into the first term for 
$t\geqslant t_{0}.$ So%
\begin{equation*}
\left\langle W\right\rangle _{t}-W(t)\geqslant \frac{C\gamma }{t}\chi
\left\{ t\geqslant t_{0}\right\} .
\end{equation*}
On the other hand,%
\begin{eqnarray*}
&&\left\langle W\right\rangle _{t}-W(t) = \frac{1}{t}\int_{0}^{t}W(s)ds-W(t)
\\
&\leqslant &\frac{1}{t}\int_{0}^{t}\left( V_{\infty }+\gamma e^{-B_{\infty
}s}+\frac{\gamma ^{p+1}A_{-}}{\left\langle s\right\rangle ^{p+d}}\right) ds
- \left( V_{\infty }+\gamma e^{-B_{0}t}+\frac{\gamma ^{p+1}A_{+}}{t^{p+d}}%
\chi \left\{ t\geqslant t_{0}+1\right\} \right) \\
&\leqslant &\frac{1}{t}\int_{0}^{t}\left( \gamma e^{-B_{\infty }s}+\frac{%
\gamma ^{p+1}A_{-}}{\left\langle s\right\rangle ^{p+d}}\right) ds \
\leqslant\ \frac{C(\gamma +\gamma ^{p+1}A_{-})}{1+t}.
\end{eqnarray*}
\end{proof}

\begin{corollary}
\label{Lemma:The class of W} For small enough $\gamma $, we have

(i) \ $\left\langle W\right\rangle _{t}>W(t)$ for all $t$.

(ii) \ $\left\langle W\right\rangle _{t}$ is a decreasing function.

(iii) \ $\left\langle W\right\rangle _{t}>\left\langle W\right\rangle _{s,t}$%
, $\forall s\in \left( 0,t\right) .$
\end{corollary}

\begin{proof}
For $t\leqslant t_{0}$, (i) follows from the assumption that $W$ is
decreasing over the interval $[0,t_{0}].$ For $t\geqslant t_{0}$, we quote
Lemma \ref{Lemma:IR<W>t-W(t)}.

(ii)%
\begin{equation*}
\frac{d}{dt}\left\langle W\right\rangle _{t}=\frac{1}{t}\left( -\left\langle
W\right\rangle _{t}+W\left( t\right) \right) <0.
\end{equation*}

(iii)%
\begin{eqnarray*}
\left\langle W\right\rangle _{s,t}-\left\langle W\right\rangle _{t} &=&\frac{%
1}{t-s}\int_{s}^{t}W\left( \tau \right) d\tau -\frac{1}{t}%
\int_{0}^{t}W\left( \tau \right) d\tau \\
&=&\frac{1}{t-s}\int_{0}^{t}W\left( \tau \right) d\tau -\frac{1}{t-s}%
\int_{0}^{s}W\left( \tau \right) d\tau -\frac{1}{t}\int_{0}^{t}W\left( \tau
\right) d\tau \\
&=&\frac{s}{t-s}\left( \frac{1}{t}\int_{0}^{t}W\left( \tau \right) d\tau -%
\frac{1}{s}\int_{0}^{s}W\left( \tau \right) d\tau \right) <0,\text{ by (ii).}
\end{eqnarray*}
\end{proof}

Since $\left\langle W\right\rangle _{s,t}$ is a continuous function of $s$
and $t$, the existence of a precollision at some time earlier than $t$
requires that the velocity satisfies 
\begin{equation}
v_{x}\in \left[ \inf\limits_{s<t}\left\langle W\right\rangle
_{s,t},\sup\limits_{s<t}\left\langle W\right\rangle _{s,t}\right] =\left[
\inf_{s<t}\left\langle W\right\rangle _{s,t},\left\langle W\right\rangle _{t}%
\right] ,  \label{condition:recollision condition}
\end{equation}%
by (iii) of Corollary \ref{Lemma:The class of W}. 
We estimate $\inf_{s<t}\left\langle W\right\rangle _{s,t}$ by 
\begin{equation}
\left\langle W\right\rangle _{s,t} \geqslant V_{\infty }-\gamma \left\langle
g\right\rangle _{s,t}\ \geqslant V_{\infty }  \label{estimate:IRinf_W_st}
\end{equation}%
since $g\le0$. It follows that 
\begin{equation}  \label{WinfEst}
W(t)-\inf_{s<t}\left\langle W\right\rangle _{s,t} \le V_{\infty }-\gamma
h(t,\gamma )-V_{\infty } =-\gamma h(t,\gamma ) =\gamma e^{-B_{\infty }t}+%
\frac{\gamma ^{p+1}A_{-}}{\left\langle t\right\rangle ^{p+d}}.
\end{equation}

\subsection{The Left Side}

In the next lemma we estimate the force on the left side of the cylinder.

\begin{lemma}
\label{Lemma:IRLeft}Let $W\in \mathcal{W}$ be defined in Definition \ref%
{def:W-IR} and let $K$ and $a_{0}$ satisfy the Assumptions A1-A4. If $k$ and 
$a_{0}$ satisfy the Irreversal Criterion, then for all sufficiently small $%
\gamma $ we have the inequalities 
\begin{equation*}
-\frac{C\left( \gamma +\gamma ^{p+1}A_{-}\right) ^{p+1}}{\left( 1+t\right)
^{p+d}}\leqslant r_{W}^{L}\left( t\right) \leqslant -c\frac{\gamma ^{p+1}}{%
t^{d+p}}\chi \left\{ t\geqslant t_{0}\right\} .
\end{equation*}
\end{lemma}

\begin{proof}
To establish upper and lower bounds of $-r_{W}^{L}\left( t\right),$ we need
upper and lower bounds of $f_{+}(t,\mathbf{x};\mathbf{v}).$ Recall the
boundary condition on the left end of the cylinder%
\begin{equation*}
f_{+}(t,\mathbf{x};\mathbf{v})=\int_{u_{x}\geqslant W\left( t\right)
}K\left( \mathbf{v}-\mathbf{i}W\left( t\right) ;\mathbf{u}-\mathbf{i}W\left(
t\right) \right) f_{-}(t,\mathbf{x};\mathbf{u})d\mathbf{u}.
\end{equation*}
Motivated by condition (\ref{condition:recollision condition}), we write the
precollision characteristic function as 
\begin{eqnarray*}
\chi _{0}(t,\mathbf{u}) &=&\chi \left\{ \mathbf{u:}\ \forall s\in (0,t),%
\text{ either }u_{x}\neq \left\langle W\right\rangle _{s,t}\text{ or }%
\left\vert u_{\perp }\right\vert >\frac{2r}{t-s}\right\} , \\
\chi _{1}(t,\mathbf{u}) &=&\chi \left\{ \mathbf{u:\exists }s\in (0,t)\text{
s.t. }u_{x}=\left\langle W\right\rangle _{s,t}\text{ and }\left\vert
u_{\perp }\right\vert \leqslant \frac{2r}{t-s}\right\} .
\end{eqnarray*}%
We observe that if the precollisions occurred at a sequence of earlier times 
$t_{j}$ converging to $t$, it would then follow that $v_{x}=W(t)$, so there
would be no contribution to the force since $\ell (0)=0$. In light of this
observation, we can always assume that there is a first precollision, that
is, a collision that occurs at an earlier time closest to $t$. In such a
case, let $\tau $ be the time and $\mathbf{\xi }$ be the position of that
first precollision. Of course, $\tau $ and $\mathbf{\xi }$ depend on $t,%
\mathbf{x,u}$. These notations enable us to write 
\begin{equation}
f_{-}(t,\mathbf{x};\mathbf{u})=f_{+}(\tau ,\mathbf{\xi };\mathbf{u})\chi
_{1}(t,\mathbf{u})+f_{0}\left( \mathbf{u}\right) \chi _{0}(t,\mathbf{u}).
\label{equation:f- with precollision}
\end{equation}%
Putting (\ref{equation:f- with precollision}) into the boundary condition,
we have 
\begin{eqnarray*}
f_{+}(t,\mathbf{x};\mathbf{v}) &=&\int_{u_{x}\geqslant W\left( t\right)
}K\left( \mathbf{v}-\mathbf{i}W\left( t\right) ;\mathbf{u}-\mathbf{i}W\left(
t\right) \right) \\
&&\times \lbrack f_{+}(\tau ,\mathbf{\xi };\mathbf{u})\chi _{1}(t,\mathbf{u}%
)+f_{0}\left( \mathbf{u}\right) \chi _{0}(t,\mathbf{u})]d\mathbf{u}.
\end{eqnarray*}
Because the momentum is only transferred horizontally, we may rewrite this
equation as 
\begin{eqnarray}
a_{+}(t,\mathbf{x;}v_{x})b(v_{\bot }) &=& b(v_{\bot })\int_{u_{x}\geqslant
W\left( t\right) }k(v_{x}-W(t),u_{x}-W(t))  \label{eqn:f_+ iterate once} \\
&&\times \{a_{+}(\tau ,\mathbf{\xi ;}u_{x})b(u_{\bot })\chi _{1}(t,\mathbf{u}%
)+a_{0}(u_{x})b(u_{\bot })\chi _{0}(t,\mathbf{u})\}d\mathbf{u}.  \notag
\end{eqnarray}%
Since $b(v_{\bot })$ could possibly vanish, we do not divide by $b(v_{\bot
}) $ on both sides. In order to get a lower bound of $f_{+}(t,\mathbf{x};%
\mathbf{v}),$ we notice that 
\begin{eqnarray*}
f_{+}(t,\mathbf{x};\mathbf{v}) &=&a_{+}(t,\mathbf{x;}v_{x})b(v_{\bot }) \\
&\geqslant &b(v_{\bot })\int_{u_{x}\geqslant W\left( t\right)
}k(v_{x}-W(t),u_{x}-W(t))a_{0}(u_{x})b(u_{\bot })\chi _{0}(t,\mathbf{u})d%
\mathbf{u}. \\
&\geqslant &b(v_{\bot })\int_{V_{\infty }+2\gamma }^{+\infty
}k(v_{x}-W(t),u_{x}-W(t))a_{0}(u_{x})du_{x},
\end{eqnarray*}%
where we used in the last line the inequality 
\begin{equation*}
W\left( t\right) \leqslant V_{\infty }+\gamma e^{-B_{\infty }t}+\frac{\gamma
^{p+1}A_{-}}{\left\langle t\right\rangle ^{p+d}}\leqslant V_{\infty
}+2\gamma.
\end{equation*}%
Thus for some $s$ we have 
\begin{equation*}
u_{x}=\left\langle W\right\rangle _{s,t}\leqslant V_{\infty }+2\gamma
\end{equation*}%
so that $\chi _{0}(t,\mathbf{u})=1$. Now recall the Irreversal Criterion 
\begin{equation*}
\int_{u_{x}\geqslant V_{\infty }}k(0,u_{x}-V_{\infty
})a_{0}(u_{x})du_{x}>a_{0}(V_{\infty }).
\end{equation*}%
It follows from A.4 that there exists $\delta >0$ such that%
\begin{equation*}
\inf_{_{\substack{ t,x\mathbf{\in \partial \Omega (}t\mathbf{)}  \\ v_{x}\in %
\left[ V_{\infty }-2\gamma ,V_{\infty }+2\gamma \right] }}%
}\int_{u_{x}\geqslant V_{\infty }+2\gamma
}k(v_{x}-W(t),u_{x}-W(t))a_{0}(u_{x})du_{x}\geqslant a_{0}(V_{\infty
})+\delta ,
\end{equation*}%
for all small enough $\gamma $. Hence we have obtained the lower bound 
\begin{equation}
f_{+}(t,\mathbf{x};\mathbf{v})\geqslant (f_{0}(\mathbf{v})+\delta b(v_\perp)
)\quad \text{ for }\left\vert v_{x}-V_{\infty }\right\vert \leqslant 2\gamma
.  \label{bound:IR lower bound of IL}
\end{equation}%
To gain an upper bound of $f_{+}(t,\mathbf{x};\mathbf{v}),$ we return to (%
\ref{eqn:f_+ iterate once}) and observe that 
\begin{eqnarray*}
a_{+}(t,\mathbf{x;}v_{x})b(v_{\bot }) &\leqslant &a_{+}^{\ast
}\int_{W(t)}^{\left\langle W\right\rangle _{t}}k(v_{x}-W(t),u_{x}-W(t))du_{x}
\\
&&+b(v_{\bot })\int_{u_{x}\geqslant W\left( t\right)
}k(v_{x}-W(t),u_{x}-W(t))a_{0}(u_{x})du_{x}
\end{eqnarray*}%
by \eqref{condition:recollision condition}, where 
\begin{equation}
a_{+}^{\ast }=\sup \left\{ a_{+}(\tau ,\xi ;u_{x})\text{ }|\text{ }\xi \in
\partial \Omega (\tau )\text{, }\tau \in \left[ 0,\infty \right) \text{, and 
}u_{x}\in \left[ V_{\infty }-2\gamma ,V_{\infty }+2\gamma \right] \right\} .
\label{def:a+*}
\end{equation}
By the fact that $\left\langle W\right\rangle _{t}-W\left( t\right)
\leqslant C\gamma ,$ proven in Lemma \ref{Lemma:IR<W>t-W(t)}, we have%
\begin{equation*}
a_{+}(t,\mathbf{x;}v_{x})b(v_{\bot })\leqslant b(v_{\bot })C\gamma
a_{+}^{\ast }+b(v_{\bot })C,
\end{equation*}%
using Assumptions {A2 }and {A4}. Thus taking the supremum over all times $%
t\in \left[ 0,\infty \right) $, positions $x\in \partial \Omega (t)$ and
velocities $v_{x}\in \left[ V_{\infty }-2\gamma ,V_{\infty }+2\gamma \right] 
$, we have 
\begin{equation*}
b(v_{\bot })a_{+}^{\ast }\leqslant b(v_{\bot })C\gamma a_{+}^{\ast
}+Cb(v_{\bot }).
\end{equation*}%
That is, 
\begin{equation}
b(v_{\bot })a_{+}^{\ast }\leqslant \frac{Cb(v_{\bot })}{1-C\gamma }\leqslant
Cb(v_{\bot }),  \label{bound:IR upper bound for f+}
\end{equation}%
for $\gamma <\frac{1}{C}$. This is an upper bound for $f_{+}(t,\mathbf{x};%
\mathbf{v}).$

We are now ready to establish upper and lower bounds of $-r_{W}^{L}\left(
t\right) $ for the irreversal case. We begin with the crucial lower bound of 
$-r_{W}^{L}\left( t\right) $ because that is the main reason why $%
r_{W}^{L}+r_{W}^{R}\leqslant 0$. Using the lower bound (\ref{bound:IR lower
bound of IL}) of $f_{+}\left( \tau ,\mathbf{\xi },\mathbf{u}\right) $, we get%
\begin{eqnarray}
-r_{W}^{L}\left( t\right) &=&\int_{\partial \Omega _{L}\left( t\right) }dS_{%
\mathbf{x}}\int_{u_{x}\geqslant W\left( t\right) }d\mathbf{u}\ \ell
(u_{x}-W(t))\left\{ f_{-}(t,\mathbf{x},\mathbf{u})-f_{0}(\mathbf{u})\right\}
\notag \\
&=&\int_{\partial \Omega _{L}\left( t\right) }dS_{\mathbf{x}%
}\int_{u_{x}\geqslant W\left( t\right) }d\mathbf{u}\ \ell
(u_{x}-W(t))[f_{+}\left( \tau ,\mathbf{\xi },\mathbf{u}\right) \chi
_{1}\left( t,\mathbf{u}\right) +f_{0}(\mathbf{u})\chi _{0}\left( t,\mathbf{u}%
\right) -f_{0}(\mathbf{u})]  \notag \\
&=&\int_{\partial \Omega _{L}\left( t\right) }dS_{\mathbf{x}%
}\int_{\left\vert u_{\bot }\right\vert \leqslant \frac{2r}{t-\tau }}du_{\bot
}\int_{W\left( t\right) }^{\left\langle W\right\rangle _{t}}du_{x}\ \ell
(u_{x}-W(t))(f_{+}\left( \tau ,\mathbf{\xi },\mathbf{u}\right) -f_{0}(%
\mathbf{u}))  \label{eqn:-RwL} \\
&\geqslant &\int_{\partial \Omega _{L}\left( t\right) }dS_{\mathbf{x}%
}\int_{\left\vert u_{\bot }\right\vert \leqslant \frac{2r}{t-\tau }}du_{\bot
}b(u_{\perp })\int_{W\left( t\right) }^{\left\langle W\right\rangle _{t}}du\
\ell (u_{x}-W(t))\left( a_{0}(V_{\infty })+\delta -a_{0}(u_{x})\right) 
\notag
\end{eqnarray}%
because $u_x = \langle W\rangle_{\tau,t} \le \langle W\rangle_t$ for some $%
\tau$. For small enough $\gamma $, by continuity of $a_{0}$ we have 
\begin{equation*}
a_{0}(V_{\infty })+\delta -a_{0}(u_{x})\geqslant \frac{\delta }{2}>0.
\end{equation*}
We then infer via A3 that 
\begin{equation*}
-r_{W}^{L}\left( t\right) \geqslant C\frac{\delta }{2}\int_{\left\vert
u_{\bot }\right\vert \leqslant \frac{2r}{t-\tau }}b(u_{\perp })du_{\bot
}\int_{W\left( t\right) }^{\left\langle W\right\rangle _{t}}du_{x}\ell
(u_{x}-W(t))\geqslant C\frac{\delta }{2}\frac{\left( \left\langle
W\right\rangle _{t}-W(t)\right) ^{p+1}}{\left( 1+t\right) ^{d-1}}\geqslant 0.
\end{equation*}
We note that in this expression the integral over $u_\perp$ is at least $C
t^{1-d}$ because $b(0)>0$. Furthermore, as proven in Lemma \ref%
{Lemma:IR<W>t-W(t)}, we have%
\begin{equation*}
\left\langle W\right\rangle _{t}-W(t)\geqslant \frac{C\gamma }{t}\chi
\left\{ t\geqslant t_{0}\right\} ,
\end{equation*}%
whence 
\begin{equation*}
-r_{W}^{L}\left( t\right) \geqslant C\frac{\gamma ^{p+1}}{t^{d+p}}\chi
\left\{ t\geqslant t_{0}\right\}
\end{equation*}%
for small enough $\gamma $. This is the desired lower bound of $-r_{W}^{L}$.

By the upper bound \eqref{bound:IR upper bound for f+} of $f_{+}\left( \tau ,%
\mathbf{\xi },\mathbf{u}\right) $ and Lemma \ref{Lemma:IR<W>t-W(t)}, we now
determine an upper bound for $-r_{W}^{L}.$ Indeed, by (\ref{eqn:-RwL}), 
\begin{eqnarray*}
\left\vert -r_{W}^{L}\left( t\right) \right\vert &=&\left\vert
\int_{\partial \Omega _{L}\left( t\right) }dS_{\mathbf{x}}\int_{W\left(
t\right) }^{\left\langle W\right\rangle _{t}}du_{x}\int_{\left\vert u_{\bot
}\right\vert \leqslant \frac{2r}{t-\tau }}du_{\bot }\ell (u_{x}-W(t))\left[
f_{+}\left( \tau ,\mathbf{\xi },\mathbf{u}\right) -f_{0}(\mathbf{u})\right]
\right\vert \\
&\leqslant &\int_{\partial \Omega _{L}\left( t\right) }dS_{\mathbf{x}%
}\int_{W\left( t\right) }^{\left\langle W\right\rangle
_{t}}du_{x}\int_{\left\vert u_{\bot }\right\vert \leqslant \frac{2r}{t-\tau }%
}du_{\bot }\ell (u_{x}-W(t))\left[ f_{+}\left( \tau ,\mathbf{\xi },\mathbf{u}%
\right) +f_{0}(\mathbf{u})\right] \\
&\leqslant &C\int_{W\left( t\right) }^{\left\langle W\right\rangle
_{t}}du_{x}\int_{\left\vert u_{\bot }\right\vert \leqslant \frac{2r}{t-\tau }%
}du_{\bot }\ell (u_{x}-W(t))\ b(u_{\bot }).
\end{eqnarray*}%
Splitting the integral according to the regions $\tau <t/2$ and $\tau \geq
t/2$, we have 
\begin{eqnarray*}
\left\vert r_{W}^{L}\left( t\right) \right\vert &\leqslant &\frac{C\left(
\left\langle W\right\rangle _{t}-W(t)\right) ^{p+1}}{\left( 1+t\right) ^{d-1}%
}+C\int_{W\left( t\right) }^{\left\langle W\right\rangle
_{t}}du_{x}\int_{\left\vert u_{\bot }\right\vert \leqslant \frac{2r}{t-\tau }%
\text{,}\tau \geqslant \frac{t}{2}}du_{\bot }\ell (u_{x}-W(t))b(u_{\bot }) \\
&=&I+II.
\end{eqnarray*}
By Assumption A3 and Lemma \ref{Lemma:IR<W>t-W(t)},%
\begin{equation*}
I\leqslant \frac{C\left( \frac{C\left( \gamma +\gamma ^{p+1}A_{-}\right) }{%
1+t}\right) ^{p+1}}{\left( 1+t\right) ^{d-1}}\leqslant C\frac{\left( \gamma
+\gamma ^{p+1}A_{-}\right) ^{p+1}}{\left( 1+t\right) ^{d+p}}.
\end{equation*}
For the second term \textit{{II}, by the precollision condition (\ref%
{condition:recollision condition}) we notice that 
\begin{eqnarray*}
u_{x} &=&\left\langle W\right\rangle _{\tau ,t}\leqslant V_{\infty }-\gamma
\left\langle h\right\rangle _{\tau ,t} \\
&\leqslant &V_{\infty }+\sup_{\frac{t}{2}\leqslant \tau \leqslant t}\frac{%
\gamma }{t-\tau }\int_{\tau }^{t}\left( e^{-B_{\infty }r}+\frac{\gamma
^{p}A_{-}}{\left\langle r\right\rangle ^{p+d}}\right) dr \\
&\leqslant &V_{\infty }+CM(t),
\end{eqnarray*}%
where 
\begin{equation*}
M(t)=\gamma e^{-B_{\infty }\frac{t}{2}}+\frac{\gamma ^{p+1}A_{-}}{%
\left\langle t\right\rangle ^{p+d}}\leqslant \frac{\gamma +\gamma ^{p+1}A_{-}%
}{\left\langle t\right\rangle ^{p+d}}.
\end{equation*}%
With Assumption A3, the above inequality allows us to estimate the second
term as 
\begin{eqnarray*}
II &\leqslant &C\int_{W(t)}^{V_{\infty }+CM(t)}du_{x}\int_{\left\vert
u_{\bot }\right\vert \leqslant \frac{2r}{t-\tau }\text{,}\tau \geqslant 
\frac{t}{2}}du_{\bot }\ell (u_{x}-W(t))\ b(u_{\bot }) \\
&\leqslant &C\int_{0}^{V_{\infty }+CM(t)-W(t)}\left\vert z\right\vert ^{p}dz
\leqslant C\left[ V_{\infty }-W(t)+CM(t)\right] ^{p+1} \leqslant C\left[ M(t)%
\right] ^{p+1} \\
&\leqslant &\ C\left( \frac{\gamma +\gamma ^{p+1}A_{-}}{\left\langle
t\right\rangle ^{p+d}}\right) ^{p+1}
\end{eqnarray*}
because $V_{\infty }-W(t)\leqslant 0$. Putting $I$ and $II$ together, we have%
\begin{equation*}
\left\vert r_{W}^{L}\left( t\right) \right\vert \leqslant C\frac{\left(
\gamma +\gamma ^{p+1}A_{-}\right) ^{p+1}}{\left( 1+t\right) ^{d+p}},
\end{equation*}%
which is the claimed upper bound for $-r_{W}^{L}\left( t\right) $. }
\end{proof}


\subsection{The Right Side\label{sec:right}}

We now proceed to bound the force $\left\vert r_{W}^{R}\right\vert $ on the
right side of the cylinder.

\begin{lemma}
\label{Lemma:IRRight}Under the same assumptions as in Lemma \ref%
{Lemma:IRLeft}, we have 
\begin{equation*}
\left\vert r_{W}^{R}\left( t\right) \right\vert \leqslant \frac{C\gamma
^{p+1}A_{-}^{p+1}}{t^{\left( p+d\right) \left( p+1\right) }}\chi \left\{
t\geqslant t_{0}\right\} .
\end{equation*}
\end{lemma}

\begin{proof}
We first notice that $r_{W}^{R}\left( t\right) =0$ for all $t\leqslant t_{0}$
because $W$ is decreasing. In fact, suppose that on the right there is a
precollision at time $\tau $ and a later collision at time $t\leq t_{0}$. If
the velocity of the particle in the time period $\left( \tau ,t\right) $ is $%
u_{x},$ then $u_{x}\geqslant W\left( \tau \right) $ and $u_{x}\leqslant
W(t)<W(\tau )$ which is a contradiction.

Taking $t\geqslant t_{0}$ and recalling the boundary condition on the right
side of the cylinder, we have 
\begin{equation*}
f_{+}(t,\mathbf{x};\mathbf{v})=\int_{u_{x}\leqslant W\left( t\right)
}K\left( \mathbf{v}-\mathbf{i}W\left( t\right) ;\mathbf{u}-\mathbf{i}W\left(
t\right) \right) f_{-}(t,\mathbf{x};\mathbf{u})d\mathbf{u}.
\end{equation*}%
Plugging in the precollision condition (\ref{equation:f- with precollision})
again, namely 
\begin{equation*}
f_{-}(t,\mathbf{x};\mathbf{u})=f_{+}(\tau ,\mathbf{\xi };\mathbf{u})\chi
_{1}(t,\mathbf{u})+f_{0}\left( \mathbf{u}\right) \chi _{0}(t,\mathbf{u}),
\end{equation*}%
we have%
\begin{eqnarray*}
f_{+}(t,\mathbf{x};\mathbf{v}) &=&\int_{u_{x}\leqslant W\left( t\right)
}K\left( \mathbf{v}-\mathbf{i}W\left( t\right) ;\mathbf{u}-\mathbf{i}W\left(
t\right) \right) \\
&&\times \left\{ f_{+}(\tau ,\mathbf{\xi };\mathbf{u})\chi _{1}(t,\mathbf{u}%
)+f_{0}\left( \mathbf{u}\right) \chi _{0}(t,\mathbf{u})\right\} d\mathbf{u}
\\
&=&b(v_{\bot })\int_{u_{x}\leqslant W\left( t\right)
}k(v_{x}-W(t),u_{x}-W(t)) \\
&&\times \left\{ a_{+}(\tau ,\mathbf{\xi ;}u_{x})b(u_{\bot })\chi _{1}(t,%
\mathbf{u})+a_{0}(u_{x})b(u_{\bot })\chi _{0}(t,\mathbf{u})\right\} d\mathbf{%
u}.
\end{eqnarray*}%
We then estimate 
\begin{eqnarray*}
f_{+}(t,\mathbf{x};\mathbf{v}) &\leqslant &b(v_{\bot
})\int_{\inf_{s<t}\left\langle W\right\rangle _{s,t}}^{W\left( t\right)
}k(v_{x}-W(t),u_{x}-W(t))a_{+}(\tau ,\mathbf{\xi ;}u_{x})du_{x} \\
&&+b(v_{\bot })\int_{-\infty }^{W\left( t\right)
}k(v_{x}-W(t),u_{x}-W(t))a_{0}(u_{x})du_{x}
\end{eqnarray*}%
Assumption A4 takes care of the second term. To estimate the first term, we
must control the size of $W\left( t\right) -\inf_{s<t}\left\langle
W\right\rangle _{s,t}.$ Recalling (\ref{WinfEst}), 
\begin{equation*}
W\left( t\right) -\inf_{s<t}\left\langle W\right\rangle _{s,t}\leqslant
\gamma e^{-B_{\infty }t}+\frac{\gamma ^{p+1}A_{-}}{\left\langle
t\right\rangle ^{p+d}},
\end{equation*}
we estimate the first term by%
\begin{equation*}
b(v_{\bot })\int_{\inf_{s<t}\left\langle W\right\rangle _{s,t}}^{W\left(
t\right) }k(v_{x}-W(t),u_{x}-W(t))a_{+}(\tau ,\mathbf{\xi ;}%
u_{x})du_{x}\leqslant b(v_{\bot })C\gamma a_{+}^{\ast }.
\end{equation*}%
where $a_{+}^{\ast }$ is defined in (\ref{def:a+*}). Thus, taking supremums
as in the earlier estimate (\ref{bound:IR upper bound for f+}), we have $%
b(v_{\bot })a_{+}^{\ast }\leqslant b(v_{\bot })C\gamma a_{+}^{\ast
}+Cb(v_{\bot }). $ Since $\gamma $ is small, we deduce that 
\begin{equation*}
b(v_{\bot })a_{+}^{\ast }\leqslant Cb(v_{\bot }).
\end{equation*}%
With this upper bound of $f_{+}(\tau ,\mathbf{\xi },u)$, we arrive at 
\begin{eqnarray*}
\left\vert r_{W}^{R}\left( t\right) \right\vert &=&\left\vert \int_{\partial
\Omega _{R}(t)}dS_{x}\int_{u_{x}\leqslant W\left( t\right) }d\mathbf{u}\
\ell (u_{x}-W(t))\left[ f_{-}(t,\mathbf{x};\mathbf{u})-f_{0}(\mathbf{u})%
\right] \right\vert \\
&=&\left\vert \int_{\partial \Omega _{R}(t)}dS_{x}\int_{u_{x}\leqslant
W\left( t\right) }d\mathbf{u}\ \ell (u_{x}-W(t))\left[ f_{+}(\tau ,\mathbf{%
\xi };\mathbf{u})\chi _{1}\left( t,\mathbf{u}\right) +f_{0}(\mathbf{u})\chi
_{0}(t,\mathbf{u})-f_{0}(\mathbf{u})\right] \right\vert \\
&=&\left\vert \int_{\partial \Omega
_{R}(t)}dS_{x}\int_{\inf_{s<t}\left\langle W\right\rangle _{s,t}}^{W(t)}d%
\mathbf{u}\ \ell (u_{x}-W(t))\left[ f_{+}(\tau ,\mathbf{\xi };\mathbf{u}%
)-f_{0}(\mathbf{u})\right] \right\vert \\
&\leqslant &\int_{\partial \Omega _{R}(t)}dS_{x}\int_{\inf_{s<t}\left\langle
W\right\rangle _{s,t}}^{W(t)}d\mathbf{u}\ \ell (u_{x}-W(t))\left[ f_{+}(\tau
,\mathbf{\xi };\mathbf{u})+f_{0}(\mathbf{u})\right] \\
&\leqslant &C\int_{\left\vert u_{\bot }\right\vert \leqslant \frac{2r}{%
t-\tau }}du_{\bot }\int_{\inf_{s<t}\left\langle W\right\rangle
_{s,t}}^{W(t)}du_{x}\ \ell (u_{x}-W(t))b(u_{\bot }).
\end{eqnarray*}%
As before, we split the integral at $\tau =t/2$. We deal with the $\tau <%
\frac{t}{2}$ part first (although it is not the main contribution unless $%
d=1 $). We have%
\begin{equation*}
I=C\int_{\inf_{s<t}\left\langle W\right\rangle _{s,t}}^{W(t)}du_{x}\
\int_{\left\vert u_{\bot }\right\vert \leqslant \frac{2r}{t-\tau },\tau <%
\frac{t}{2}}du_{\bot }\ell (u_{x}-W(t))b(u_{\bot })\leqslant C\frac{\left(
W(t)-\inf_{s<t}\left\langle W\right\rangle _{s,t}\right) ^{p+1}}{\left(
1+t\right) ^{d-1}}
\end{equation*}%
%
%
%
%
%
\begin{equation*}
\leqslant \frac{C}{\left( 1+t\right) ^{d-1}}\left( \gamma e^{-B_{\infty }t}+%
\frac{\gamma ^{p+1}A_{-}}{\left\langle t\right\rangle ^{p+d}}\right) ^{p+1}
\end{equation*}%
for $t\geqslant t_{0}$ by \eqref{WinfEst}.

For the $\tau \in \left[ \frac{t}{2},t\right] $ part, which is the major
contribution, we know 
\begin{eqnarray*}
u_{x} &=&\left\langle W\right\rangle _{\tau ,t}\geqslant V_{\infty }-\gamma
\left\langle g\right\rangle _{\tau ,t} \\
&\geqslant &V_{\infty }+\inf_{\frac{t}{2}\leqslant \tau \leqslant t}\frac{%
\gamma }{t-\tau }\int_{\tau }^{t}\left( e^{-B_{0}r}+\frac{\gamma ^{p}A_{+}}{%
r^{p+d}}\chi \left\{ r\geqslant t_{0}+1\right\} \right) dr \\
&\geqslant &V_{\infty }
\end{eqnarray*}
which yields, for $t\geqslant t_{0}$, 
\begin{eqnarray*}
II &\leqslant &C\int_{V_{\infty }}^{W(t)}du_{x}\ \int_{\left\vert u_{\bot
}\right\vert \leqslant \frac{2r}{t-\tau },\frac{t}{2}\leqslant \tau
\leqslant t}du_{\bot }\ell (u_{x}-W(t))b(u_{\bot }) \\
&\leqslant &C\int_{V_{\infty }-W(t)}^{0}\left\vert u_{x}-W(t)\right\vert
^{p}du_{x}\ \leqslant\ C\left\vert V_{\infty }-W(t)\right\vert ^{p+1} \\
&\leqslant &C\left\vert \gamma e^{-B_{\infty }t}+\frac{\gamma ^{p+1}A_{-}}{%
\left\langle t\right\rangle ^{p+d}}\right\vert ^{p+1},
\end{eqnarray*}%
by (\ref{g and h}). Collecting the estimates for $I$ and $II$, we have%
\begin{eqnarray*}
\left\vert r_{W}^{R}\left( t\right) \right\vert &\leqslant &\frac{C}{\left(
1+t\right) ^{d-1}}\left( \gamma e^{-B_{\infty }t}+\frac{\gamma ^{p+1}A_{-}}{%
\left\langle t\right\rangle ^{p+d}}\right) ^{p+1}\chi \left\{ t\geqslant
t_{0}\right\} \\
&&+C\left( \gamma e^{-B_{\infty }t}+\frac{\gamma ^{p+1}A_{-}}{\left\langle
t\right\rangle ^{p+d}}\right) ^{p+1}\chi \left\{ t\geqslant t_{0}\right\} \\
&\leqslant &\frac{C\gamma ^{p+1}A_{-}^{p+1}}{\left\langle t\right\rangle
^{\left( p+d\right) \left( p+1\right) }}\chi \left\{ t\geqslant t_{0}\right\}
\end{eqnarray*}%
since $d\geqslant 1$. This concludes the proof of Lemma \ref{Lemma:IRRight}.
\end{proof}

Finally, collecting Lemmas \ref{Lemma:IRLeft} and \ref{Lemma:IRRight}, we
have both%
\begin{equation*}
R_{W}(t)\leqslant \left[ -c\frac{\gamma ^{p+1}}{t^{p+d}}+C_{2}\frac{\gamma
^{p+1}A_{-}^{p+1}}{\left\langle t\right\rangle ^{\left( p+d\right) \left(
p+1\right) }}\right] \chi \left\{ t\geqslant t_{0}\right\}
\end{equation*}%
and%
\begin{eqnarray*}
R_{W}(t) &\geqslant &-\frac{C_{1}\left( \gamma +\gamma ^{p+1}A_{-}\right)
^{p+1}}{\left\langle t\right\rangle ^{p+d}}-C_{2}\frac{\gamma
^{p+1}A_{-}^{p+1}}{\left\langle t\right\rangle ^{\left( p+d\right) \left(
p+1\right) }}\chi \left\{ t\geqslant t_{0}\right\} \\
&\geqslant &-\frac{C\left( \gamma +\gamma ^{p+1}A_{-}\right) ^{p+1}}{%
\left\langle t\right\rangle ^{p+d}}
\end{eqnarray*}%
Because $\gamma $ is small, Theorem \ref{Thm:IREstimate} follows.


\section{Proof of the Reversal Case\label{Sec:R}}

\subsection{Proof assuming the Key Estimate}

For the proof of the reversal case, we follow the structure of the proof of
the irreversal case in Section \ref{Sec:IR}. Alert reader should keep in
mind that we use a class $\mathcal{W}$ different from Definition \ref%
{def:W-IR} here for the reversal case. To be specific, the definitions of $%
t_{0}$, $g$, and $h$ are different.

\begin{definition}[Class of possible motions for the reversal case]
\label{def:W-R}We define $\mathcal{W}$ as the family of functions $W$ which
satisfy the following conditions.

(i) $W:[0,\infty )\rightarrow \mathbb{R}$ is Lipschitz and $W(0)=V_{\infty
}+\gamma $.

(ii) $W$ is decreasing over the interval $[0,t_{0}]$ for $%
t_{0}=K_{0}\left\vert \ln \gamma \right\vert $ with $\frac{1}{B_{0}}%
\leqslant K_{0}\leqslant \frac{2}{B_{0}}$, where 
\begin{equation*}
B_{0}=\max_{V\in \left[ V_{\infty }-\gamma ,V_{\infty }+\gamma \right]
}F_{0}^{\prime }(V),B_{\infty }=\min_{V\in \left[ V_{\infty }-\gamma
,V_{\infty }+\gamma \right] }F_{0}^{\prime }(V).
\end{equation*}

(iii) For all $W\in \mathcal{W}$, $t\in \lbrack 0,\infty )$ and $\gamma \in
(0,1)$, 
\begin{equation}
\gamma h(t,\gamma )\leqslant V_{\infty }-W(t)\leqslant \gamma g(t,\gamma ),
\label{r:g and h}
\end{equation}%
that is, 
\begin{equation*}
V_{\infty }-\gamma g(t,\gamma )\leqslant W(t)\leqslant V_{\infty }-\gamma
h(t,\gamma ),
\end{equation*}%
where%
\begin{eqnarray*}
-g(t,\gamma ) &=&e^{-B_{0}t}-\frac{\gamma ^{p}A_{+}}{\left\langle
t\right\rangle ^{p+d}}, \\
-h(t,\gamma ) &=&e^{-B_{\infty }t}-\frac{\gamma ^{p}A_{-}}{t^{p+d}}\chi
\left\{ t\geqslant t_{0}+1\right\} .
\end{eqnarray*}
\end{definition}

\begin{theorem}
\label{Thm:REstimate}Assume that $k$ and $a_{0}$ verify the Reversal
Criterion, then for small enough $\gamma $, there exists $c_{1}$, $C_{2},$
and $C>0$ such that for all $W\in \mathcal{W}$, we have%
\begin{equation}
R_{W}(t)\geqslant \left( \frac{c_{1}\gamma ^{p+1}}{t^{p+d}}-\frac{%
C_{2}\gamma ^{\left( 1+\varepsilon \right) \left( p+1\right) }A_{+}^{p+1}}{%
t^{p+d}}\right) \chi \left\{ t\geqslant t_{0}\right\} \geqslant 0,
\label{R:RestimateL}
\end{equation}%
and%
\begin{equation}
R_{W}(t)\leqslant \frac{C\left( \gamma +\gamma ^{p+1}A_{+}\right) ^{p+1}}{%
\left( 1+t\right) ^{p+d}}.  \label{R:RestimateU}
\end{equation}
\end{theorem}

\begin{proof}
We postpone the proof of Theorem \ref{Thm:REstimate} to the end of Section 4.
\end{proof}


\begin{lemma}
\label{R:DeducingConditionsOnhAndg}If $k$ and $a_{0}$ satisfy the Reversal
Criterion, then for small enough $\gamma $, we can choose $A_{+}$ and $A_{-}$
in Definition \ref{def:W-R} such that, for any $W\in \mathcal{W}$, the
solution $V_{W}$ to the iteration equation (\ref{iteration equation}) 
\begin{equation*}
\frac{dV_{W}}{dt}=Q(t)\left( V_{\infty }-V_{W}\right) -R_{W}\left( t\right)
,\quad Q(t)=\frac{F_{0}(V_{\infty })-F_{0}(W(t))}{V_{\infty }-W(t)},
\end{equation*}%
satisfies 
\begin{equation*}
-\gamma e^{-B_{\infty }t}+\frac{A_{-}\gamma ^{p+1}\chi \{t\geqslant t_{0}+1\}%
}{t^{p+d}}\leqslant V_{\infty }-V_{W}(t)\leqslant -\gamma e^{-B_{0}t}+\frac{%
A_{+}\gamma ^{p+1}}{\left( 1+t\right) ^{p+d}}.
\end{equation*}%
In other words, for every $W\in \mathcal{W},$ we have $V_{W}\in \mathcal{W}.$
\end{lemma}

\begin{proof}
By \eqref{iteration equation} we have%
\begin{equation*}
\frac{d\left( V_{\infty }-V_{W}\right) }{dt}=-Q(t)\left( V_{\infty
}-V_{W}\right) +R_{W}\left( t\right) ,
\end{equation*}%
hence%
\begin{equation*}
V_{\infty }-V_{W}\left( t\right) =-\gamma
e^{-\int_{0}^{t}Q(r)dr}+\int_{0}^{t}\left( e^{-\int_{s}^{t}Q(r)dr}\right)
R_{W}(s)ds
\end{equation*}%
because $V_{\infty }-V_{W}\left( 0\right) =-\gamma $. On the one hand, by
estimate (\ref{R:RestimateU}) we have 
\begin{eqnarray*}
V_{\infty }-V_{W}\left( t\right) &\leqslant &-\gamma
e^{-B_{0}t}+\int_{0}^{t}e^{-B_{\infty }(t-s)}\frac{C\left( \gamma +\gamma
^{p+1}A_{+}\right) ^{p+1}}{\left( 1+s\right) ^{p+d}}ds \\
&=&-\gamma e^{-B_{0}t}+C\left( \gamma +\gamma ^{p+1}A_{+}\right) ^{p+1}I,
\end{eqnarray*}%
where%
\begin{eqnarray*}
I &=&\int_{0}^{\frac{t}{2}}e^{-B_{\infty }(t-s)}\frac{1}{\left( 1+s\right)
^{p+d}}ds+\int_{\frac{t}{2}}^{t}e^{-B_{\infty }(t-s)}\frac{1}{\left(
1+s\right) ^{p+d}}ds \\
&\leqslant &\int_{0}^{\frac{t}{2}}e^{-B_{\infty }(t-s)}ds+\frac{C}{\left(
1+t\right) ^{p+d}}\int_{\frac{t}{2}}^{t}e^{-B_{\infty }(t-s)}ds \\
&\leqslant &\frac{1}{B_{\infty }}(e^{-\frac{B_{\infty }t}{2}}-e^{-B_{\infty
}t})+\frac{C}{\left( 1+t\right) ^{p+d}}\ \ \leqslant \ \frac{C^{\prime }}{%
\left( 1+t\right) ^{p+d}}.
\end{eqnarray*}%
That is, 
\begin{equation*}
V_{\infty }-V_{W}\left( t\right) \leqslant -\gamma e^{-B_{0}t}+\frac{%
C^{\prime }\left( \gamma +\gamma ^{p+1}A_{+}\right) ^{p+1}}{\left(
1+t\right) ^{p+d}}.
\end{equation*}%
Letting $A_{+}>C^{\prime }$, we have $C^{\prime }\left( 1+\gamma
^{p}A_{+}\right) ^{p+1}\leqslant A_{+}$ for small $\gamma ,$ so that 
\begin{equation*}
V_{\infty }-V_{W}\left( t\right) \leqslant -\gamma e^{-B_{0}t}+\frac{%
A_{+}\gamma ^{p+1}}{\left( 1+t\right) ^{p+d}}.
\end{equation*}

On the other hand, with the aforementioned $A_{+}$, estimate (\ref%
{R:RestimateL}) reads 
\begin{equation*}
R_{W}(t)\geqslant \frac{c\gamma ^{p+1}}{t^{p+d}}\chi \left\{ t\geqslant
t_{0}\right\} \geqslant 0
\end{equation*}%
for small $\gamma $. Thus 
\begin{eqnarray*}
V_{\infty }-V_{W}\left( t\right) &\geqslant &-\gamma e^{-B_{\infty
}t}+\int_{0}^{t}e^{-B_{0}(t-s)}\chi \{s\geqslant t_{0}\}\frac{c\gamma ^{p+1}%
}{s^{p+d}}ds \\
&=&-\gamma e^{-B_{\infty }t}+c\gamma ^{p+1}II,
\end{eqnarray*}%
where%
\begin{equation*}
II\geqslant \int_{t-1}^{t}e^{-B_{0}}\frac{1}{s^{p+d}}ds\geqslant e^{-B_{0}}%
\frac{1}{t^{p+d}},
\end{equation*}%
as long as $1+t_{0}<t$. Hence%
\begin{equation*}
V_{\infty }-V_{W}\left( t\right) \geqslant -\gamma e^{-B_{\infty }t}+\frac{%
c^{\prime p+1}\chi \{t\geqslant t_{0}+1\}}{t^{p+d}}.
\end{equation*}%
Therefore, selecting $A_{-}\leqslant c^{\prime }$ yields%
\begin{equation*}
V_{\infty }-V_{W}\left( t\right) \geqslant -\gamma e^{-B_{\infty }t}+\frac{%
A_{-}\gamma ^{p+1}\chi \{t\geqslant t_{0}+1\}}{t^{p+d}}.
\end{equation*}
\end{proof}

Theorem \ref{ThRExistence} then follows by the same proof as in Theorem \ref%
{ThIRExistence}. For the reversal case, we are left with the proof of
Theorem \ref{Thm:REstimate}.



\subsection{Properties of $\mathcal{W\label{sec:R:properties of family}}$}

\begin{lemma}
\label{Lemma:R<W>t-W(t)}Let $\mathcal{W}$ be defined in Definiton \ref%
{def:W-R} for the reversal case. For all small enough $\gamma $ and hence
for all large enough $t_{0}$, we have%
\begin{equation*}
\left\langle W\right\rangle _{t}-W(t)\geqslant \frac{C\gamma }{t}\text{ for }%
t\geqslant t_{0},
\end{equation*}%
and%
\begin{equation*}
\left\langle W\right\rangle _{t}-W(t)\leqslant \frac{C\left( \gamma +\gamma
^{p+1}A_{+}\right) }{1+t}\text{, for all }t\geq 0.
\end{equation*}
\end{lemma}

\begin{proof}
On the one hand,%
\begin{eqnarray*}
\left\langle W\right\rangle _{t}-W(t) &=& \frac{1}{t}\int_{0}^{t}W(s)ds-W(t)
\\
&\geqslant &\frac{1}{t}\int_{0}^{t}\left( V_{\infty }+\gamma e^{-B_{0}s}-%
\frac{\gamma ^{p+1}A_{+}}{\left\langle s\right\rangle ^{p+d}}\right)
ds-V_{\infty }-\gamma e^{-B_{\infty }t} \\
&\geqslant &\frac{C_{1}\gamma }{t}-\frac{C_{2}\gamma ^{p+1}}{t}-\gamma
e^{-B_{\infty }t}
\end{eqnarray*}%
For all small enough $\gamma $, the second term is absorbed into the first
term because $p>0$. Also, for all small enough $\gamma $ and hence all large
enough $t_{0}$, the third term is absorbed into the first term for $%
t\geqslant t_{0}.$

On the other hand,%
\begin{eqnarray*}
\left\langle W\right\rangle _{t}-W(t) &=&\frac{1}{t}\int_{0}^{t}W(s)ds-W(t)
\\
&\leqslant &\frac{1}{t}\int_{0}^{t}\left( V_{\infty }+\gamma e^{-B_{\infty
}s}-\frac{\gamma ^{p+1}A_{-}}{s^{p+d}}\chi _{s\geqslant t_{0}+1}\right)
ds-\left( V_{\infty }+\gamma e^{-B_{0}t}-\frac{\gamma ^{p+1}A_{+}}{%
\left\langle t\right\rangle ^{p+d}}\right) \\
&\leqslant &\frac{1}{t}\int_{0}^{t}\gamma e^{-B_{\infty }s}ds+\frac{\gamma
^{p+1}A_{+}}{\left\langle t\right\rangle ^{p+d}}\leqslant \frac{C\left(
\gamma +\gamma ^{p+1}A_{+}\right) }{1+t}.
\end{eqnarray*}
\end{proof}

\begin{corollary}
\label{Lemma:Rclass of W} For small enough $\gamma $, we have

(i) \ $\left\langle W\right\rangle _{t}>W(t)$ for all $t$.

(ii) \ $\left\langle W\right\rangle _{t}$ is a decreasing function.

(iii) \ $\left\langle W\right\rangle _{t}>\left\langle W\right\rangle _{s,t}$%
, $\forall s\in \left( 0,t\right) .$
\end{corollary}

\begin{proof}
Same as Corollary \ref{Lemma:The class of W}.
\end{proof}

For the right side estimate, we estimate $\inf_{s<t}\left\langle
W\right\rangle _{s,t}$ 
\begin{eqnarray}
\left\langle W\right\rangle _{s,t} &\geqslant &V_{\infty }-\gamma
\left\langle g\right\rangle _{s,t}\ \geqslant \ V_{\infty }+\frac{\gamma }{%
t-s}\int_{s}^{t}-\frac{\gamma ^{p}A_{+}}{\left\langle r\right\rangle ^{p+d}}%
dr  \label{estimate:inf_wW_st} \\
&\geqslant &V_{\infty }-\gamma ^{p+1}A_{+}\sup_{s<t}\left( \frac{1}{t-s}%
\int_{s}^{t}\frac{1}{\left\langle r\right\rangle ^{p+d}}dr\right)  \notag \\
&\geqslant &V_{\infty }-\frac{C\gamma ^{p+1}A_{+}}{\left\langle
t\right\rangle },  \notag
\end{eqnarray}%
where $C=\int_{0}^{\infty }\frac{1}{\left\langle r\right\rangle ^{p+d}}dr$
since $p+d>1$. with the same method, we know that 
\begin{eqnarray}
\left\langle W\right\rangle _{s,t} &\leqslant &V_{\infty }-\gamma
\left\langle h\right\rangle _{s,t}  \label{estimate:sup_wW_st} \\
&\leqslant &V_{\infty }+\frac{\gamma }{t-s}\int_{s}^{t}\left( e^{-B_{\infty
}r}-\frac{\gamma ^{p}A_{-}}{r^{p+d}}\chi \left\{ r\geqslant t_{0}+1\right\}
\right) dr  \notag \\
&\leqslant &V_{\infty }+\gamma .  \notag
\end{eqnarray}%
Examining Definition \ref{def:W-R}, we also notice that 
\begin{equation*}
V_{\infty }-\frac{\gamma ^{p+1}A_{+}}{\left\langle t\right\rangle }\leqslant
W(t)\leqslant V_{\infty }+\gamma .
\end{equation*}

\subsection{The Left Side\label{sec:Rleft}}

\begin{lemma}
\label{Lemma:Upper and lower bound of R+}Let $W\in \mathcal{W}$ and let $K$
and $a_{0}$ satisfy the Assumptions A1-A4. If $k$ and $a_{0}$ satisfy the
Reversal Criterion, then for all sufficiently small $\gamma $ we have the
inequalities 
\begin{equation*}
c\frac{\gamma ^{p+1}}{t^{d+p}}\chi \left\{ t\geqslant t_{0}\right\}
\leqslant r_{W}^{L}\left( t\right) \leqslant \frac{C\left( \gamma +\gamma
^{p+1}A_{+}\right) ^{p+1}}{\left( 1+t\right) ^{p+d}}.
\end{equation*}
\end{lemma}

\begin{proof}
We will first prove a careful upper bound of $f_{+}(t,\mathbf{x};\mathbf{v}%
). $ Recall the boundary condition on the left of the cylinder%
\begin{equation*}
f_{+}(t,\mathbf{x};\mathbf{v})=\int_{u_{x}\geqslant W\left( t\right)
}K\left( \mathbf{v}-\mathbf{i}W\left( t\right) ;\mathbf{u}-\mathbf{i}W\left(
t\right) \right) f_{-}(t,\mathbf{x};\mathbf{u})d\mathbf{u}.
\end{equation*}%
and the precollision characteristic functions 
\begin{eqnarray*}
\chi _{0}(t,\mathbf{u}) &=&\chi \left\{ \mathbf{u:}\ \forall s\in (0,t),%
\text{ either }u_{x}\neq \left\langle W\right\rangle _{s,t}\text{ or }%
\left\vert u_{\perp }\right\vert >\frac{2r}{t-s}\right\} , \\
\chi _{1}(t,\mathbf{u}) &=&\chi \left\{ \mathbf{u:\exists }s\in (0,t)\text{
s.t. }u_{x}=\left\langle W\right\rangle _{s,t}\text{ and }\left\vert
u_{\perp }\right\vert \leqslant \frac{2r}{t-s}\right\} .
\end{eqnarray*}%
We again write 
\begin{equation*}
f_{-}(t,\mathbf{x};\mathbf{u})=f_{+}(\tau ,\mathbf{\xi };\mathbf{u})\chi
_{1}(t,\mathbf{u})+f_{0}\left( \mathbf{u}\right) \chi _{0}(t,\mathbf{u}),
\end{equation*}
which gives 
\begin{eqnarray*}
f_{+}(t,\mathbf{x};\mathbf{v}) &=&\int_{u_{x}\geqslant W\left( t\right)
}K\left( \mathbf{v}-\mathbf{i}W\left( t\right) ;\mathbf{u}-\mathbf{i}W\left(
t\right) \right) \lbrack f_{+}(\tau ,\mathbf{\xi };\mathbf{u})\chi _{1}(t,%
\mathbf{u})+f_{0}\left( \mathbf{u}\right) \chi _{0}(t,\mathbf{u})]d\mathbf{u}%
.
\end{eqnarray*}%
That is, 
\begin{eqnarray*}
a_{+}(t,\mathbf{x;}v_{x})b(v_{\bot }) &=&b(v_{\bot })\int_{u_{x}\geqslant
W\left( t\right) }k(v_{x}-W(t),u_{x}-W(t)) \\
&&\times \{a_{+}(\tau ,\mathbf{\xi ;}u_{x})b(u_{\bot })\chi _{1}(t,\mathbf{u}%
)+a_{0}(u_{x})b(u_{\bot })\chi _{0}(t,\mathbf{u})\}d\mathbf{u}.
\end{eqnarray*}%
We then notice by Lemma \ref{Lemma:Rclass of W} (iii) that 
\begin{eqnarray*}
a_{+}(t,\mathbf{x;}v_{x})b(v_{\bot }) &\leqslant &b(v_{\bot })a_{+}^{\ast
}\int_{W(t)}^{\left\langle W\right\rangle _{t}}k(v_{x}-W(t),u_{x}-W(t))du_{x}
\\
&&+b(v_{\bot })\int_{u_{x}\geqslant W\left( t\right)
}k(v_{x}-W(t),u_{x}-W(t))a_{0}(u_{x})du_{x} \\
&\leqslant &b(v_{\bot })C\gamma a_{+}^{\ast }+b(v_{\bot
})\int_{u_{x}\geqslant W\left( t\right)
}k(v_{x}-W(t),u_{x}-W(t))a_{0}(u_{x})du_{x}
\end{eqnarray*}%
by A2 and Lemma \ref{Lemma:R<W>t-W(t)}, where $a_{+}^{\ast }$ is defined as
in (\ref{def:a+*}). Now recall the Reversal Criterion 
\begin{equation*}
\int_{u_{x}\geqslant V_{\infty }}k(0,u_{x}-V_{\infty
})a_{0}(u_{x})du_{x}<a_{0}(V_{\infty }).
\end{equation*}
Hence by A4 and continuity, $\exists \delta >0$ such that%
\begin{equation*}
\sup_{_{\substack{ t,x\mathbf{\in \partial \Omega (}t\mathbf{)}  \\ v_{x}\in %
\left[ V_{\infty }-2\gamma ,V_{\infty }+2\gamma \right] }}%
}\int_{u_{x}\geqslant W\left( t\right)
}k(v_{x}-W(t),u_{x}-W(t))a_{0}(u_{x})du_{x}\leqslant a_{0}(V_{\infty
})-\delta
\end{equation*}%
for all $\gamma $ small enough. We thus arrive at 
\begin{equation*}
a_{+}(t,\mathbf{x;}v_{x})b(v_{\bot })\leqslant b(v_{\bot })C\gamma
a_{+}^{\ast }+b(v_{\bot })\left( a_{0}(V_{\infty })-\delta \right) .
\end{equation*}%
That is,%
\begin{equation}
b(v_{\bot })a_{+}^{\ast }\leqslant \frac{1}{1-C\gamma }\left(
a_{0}(V_{\infty })-\delta \right) b(v_{\bot })\leqslant b(v_{\bot })\left(
a_{0}(V_{\infty })-\frac{\delta }{2}\right)  \label{Lemma:RfUpperBound}
\end{equation}%
for small $\gamma .$ Now recall that 
\begin{eqnarray*}
r_{W}^{L}\left( t\right) &=&\int_{\partial \Omega _{L}\left( t\right) }dS_{%
\mathbf{x}}\int_{u_{x}\geqslant W\left( t\right) }d\mathbf{u}\ \ell
(u_{x}-W(t))\left\{ f_{0}(\mathbf{u})-f_{-}(t,\mathbf{x},\mathbf{u})\right\}
\\
&=&\int_{\partial \Omega _{L}\left( t\right) }dS_{\mathbf{x}%
}\int_{u_{x}\geqslant W\left( t\right) }d\mathbf{u}\ \ell (u_{x}-W(t))[f_{0}(%
\mathbf{u})-f_{+}\left( \tau ,\mathbf{\xi },\mathbf{u}\right) \chi
_{1}\left( t,\mathbf{u}\right) -f_{0}(\mathbf{u})\chi _{0}\left( t,\mathbf{u}%
\right) ] \\
&=&\int_{\partial \Omega _{L}\left( t\right) }dS_{\mathbf{x}%
}\int_{\left\vert u_{\bot }\right\vert \leqslant \frac{2r}{t-\tau }}du_{\bot
}\int_{W\left( t\right) }^{\left\langle W\right\rangle _{t}}du_{x}\ \ell
(u_{x}-W(t))(f_{0}(\mathbf{u})-f_{+}\left( \tau ,\mathbf{\xi },\mathbf{u}%
\right) ).
\end{eqnarray*}
With (\ref{Lemma:RfUpperBound}) we obtain 
\begin{equation*}
r_{W}^{L}\left( t\right) \geqslant \int_{\partial \Omega _{L}\left( t\right)
}dS_{\mathbf{x}}\int_{\left\vert u_{\bot }\right\vert \leqslant \frac{2r}{%
t-\tau }}b(v_{\bot })du_{\bot }\int_{W\left( t\right) }^{\left\langle
W\right\rangle _{t}}du_{x}\ \ell (u_{x}-W(t))(a_{0}(u_{x})-a_{0}(V_{\infty
})+\frac{\delta }{2}).
\end{equation*}%
For small enough $\gamma $, we have by continuity of $a_{0}$ that%
\begin{equation*}
a_{0}(u_{x})-a_{0}(V_{\infty })+\frac{\delta }{2}\geqslant \frac{\delta }{4}%
>0,
\end{equation*}%
which implies 
\begin{eqnarray*}
r_{W}^{L}\left( t\right) &\geqslant &C\frac{\delta }{4}\int_{\left\vert
u_{\bot }\right\vert \leqslant \frac{2r}{t-\tau }}b(u_{\perp })du_{\bot
}\int_{W\left( t\right) }^{\left\langle W\right\rangle _{t}}du_{x}\ell
(u_{x}-W(t)) \\
&\geqslant &C\frac{\delta }{4}\frac{\left( \left\langle W\right\rangle
_{t}-W(t)\right) ^{p+1}}{\left( 1+t\right) ^{d-1}}
\end{eqnarray*}
Recalling Lemma \ref{Lemma:R<W>t-W(t)}, we have%
\begin{equation*}
\left\langle W\right\rangle _{t}-W(t)\geqslant \frac{C\gamma }{t}\chi
\left\{ t\geqslant t_{0}\right\} ,
\end{equation*}%
whence 
\begin{equation*}
r_{W}^{L}\left( t\right) \geqslant C\frac{\gamma ^{p+1}}{t^{d+p}}\chi
\left\{ t\geqslant t_{0}\right\}
\end{equation*}%
for small enough $\gamma $. This is the desired lower bound of $r_{W}^{L}$.

For the upper bound of $r_{W}^{L},$ we have%
\begin{eqnarray*}
\left\vert r_{W}^{L}\left( t\right) \right\vert &=&\left\vert \int_{\partial
\Omega _{L}\left( t\right) }dS_{\mathbf{x}}\int_{\left\vert u_{\bot
}\right\vert \leqslant \frac{2r}{t-\tau }}du_{\bot }\int_{W\left( t\right)
}^{\left\langle W\right\rangle _{t}}du_{x}\ \ell (u_{x}-W(t))(f_{0}(\mathbf{u%
})-f_{+}\left( \tau ,\mathbf{\xi },\mathbf{u}\right) )\right\vert \\
&\leqslant &\int_{\partial \Omega _{L}\left( t\right) }dS_{\mathbf{x}%
}\int_{\left\vert u_{\bot }\right\vert \leqslant \frac{2r}{t-\tau }}du_{\bot
}\int_{W\left( t\right) }^{\left\langle W\right\rangle _{t}}du_{x}\ \ell
(u_{x}-W(t))\left\vert f_{0}(\mathbf{u})-f_{+}\left( \tau ,\mathbf{\xi },%
\mathbf{u}\right) \right\vert \\
&\leqslant &C\int_{\partial \Omega _{L}\left( t\right) }dS_{\mathbf{x}%
}\int_{W\left( t\right) }^{\left\langle W\right\rangle _{t}}du_{x}\
\int_{\left\vert u_{\bot }\right\vert \leqslant \frac{2r}{t-\tau }}du_{\bot
}\ell (u_{x}-W(t))b(u_{\bot }).
\end{eqnarray*}%
Splitting the integral according to the regions $\tau <t/2$ and $\tau \geq
t/2$, we have 
\begin{eqnarray*}
r_{W}^{L}\left( t\right) &\leqslant &\frac{C\left( \left\langle
W\right\rangle _{t}-W(t)\right) ^{p+1}}{\left( 1+t\right) ^{d-1}}%
+C\int_{W\left( t\right) }^{\left\langle W\right\rangle
_{t}}du_{x}\int_{\left\vert u_{\bot }\right\vert \leqslant \frac{2r}{t-\tau }%
\text{,}\tau \geqslant \frac{t}{2}}du_{\bot }\ell (u_{x}-W(t))b(u_{\bot }) \\
&=&I+II.
\end{eqnarray*}%
By Assumption A3 and Lemma \ref{Lemma:R<W>t-W(t)},%
\begin{equation*}
I\leqslant \frac{C\left( \frac{C\left( \gamma +\gamma ^{p+1}A_{+}\right) }{%
1+t}\right) ^{p+1}}{\left( 1+t\right) ^{d-1}}\leqslant C\frac{\left( \gamma
+\gamma ^{p+1}A_{+}\right) ^{p+1}}{\left( 1+t\right) ^{d+p}}.
\end{equation*}
For the second term, by the precollision condition (\ref%
{condition:recollision condition}), we notice that 
\begin{eqnarray*}
u_{x} &=&\left\langle W\right\rangle _{\tau ,t}\ \leqslant \ V_{\infty
}-\gamma \left\langle h\right\rangle _{\tau ,t} \\
&\leqslant &V_{\infty }+\sup_{\frac{t}{2}\leqslant \tau \leqslant t}\frac{%
\gamma }{t-\tau }\int_{\tau }^{t}\left( e^{-B_{\infty }r}-\frac{\gamma
^{p}A_{-}}{r^{p+d}}\chi \left\{ r\geqslant t_{0}+1\right\} \right) dr \\
&\leqslant &V_{\infty }+C\gamma e^{-B_{\infty }\frac{t}{2}}.
\end{eqnarray*}%
By Assumption A3 again, this inequality allows us to estimate the second
term as 
\begin{eqnarray*}
II &\leqslant &C\int_{W\left( t\right) }^{V_{\infty }+C\gamma e^{-B_{\infty }%
\frac{t}{2}}}du_{x}\int_{\left\vert u_{\bot }\right\vert \leqslant \frac{2r}{%
t-\tau }\text{,}\tau \geqslant \frac{t}{2}}du_{\bot }\ell (u_{x}-W(t))\
b(u_{\bot }) \\
&\leqslant &C\int_{0}^{V_{\infty }-W(t)+C\gamma e^{-B_{\infty }\frac{t}{2}%
}}dz\left\vert z\right\vert ^{p} \\
&\leqslant &C\left( V_{\infty }-W(t)+C\gamma e^{-B_{\infty }\frac{t}{2}%
}\right) ^{p+1}.
\end{eqnarray*}%
By definition of $\mathcal{W}$, we have 
\begin{equation*}
V_{\infty }-W(t)\leqslant -\gamma e^{-B_{0}t}+\frac{\gamma ^{p+1}A_{+}}{%
\left\langle t\right\rangle ^{p+d}}\leqslant \frac{\gamma ^{p+1}A_{+}}{%
\left\langle t\right\rangle ^{p+d}},
\end{equation*}%
so that 
\begin{equation*}
II\leqslant \left( \frac{\gamma ^{p+1}A_{+}}{\left\langle t\right\rangle
^{p+d}}+C\gamma e^{-B_{\infty }\frac{t}{2}}\right) ^{p+1}\leqslant \ C\left( 
\frac{C\gamma +\gamma ^{p+1}A_{+}}{\left\langle t\right\rangle ^{p+d}}%
\right) ^{p+1}.
\end{equation*}%
Putting $I$ and $II$ together, we have%
\begin{equation*}
r_{W}^{L}\left( t\right) \leqslant C\frac{\left( \gamma +\gamma
^{p+1}A_{+}\right) ^{p+1}}{\left( 1+t\right) ^{d+p}},
\end{equation*}%
which is the claimed upper bound for $r_{W}^{L}\left( t\right) $.
\end{proof}


\subsection{The Right Side\label{sec:Rright}}

\begin{lemma}
\label{Lemma:UpperBoundOfR-}Under the same assumptions as in Lemma \ref%
{Lemma:Upper and lower bound of R+}, for $\gamma $ small enough, there is a $%
\varepsilon >0$ such that 
\begin{equation*}
\left\vert r_{W}^{R}\left( t\right) \right\vert \leqslant \frac{C\gamma
^{\left( 1+\varepsilon \right) \left( p+1\right) }A_{+}^{p+1}}{t^{p+d}}\chi
\left\{ t\geqslant t_{0}\right\} .
\end{equation*}
\end{lemma}

\begin{proof}
As in the begining of the proof of Lemma \ref{Lemma:IRRight}, we have $%
r_{W}^{R}\left( t\right) =0$ for all $t\leqslant t_{0}$ because $W$ is
decreasing. Setting $t\geqslant t_{0}$ and recalling the boundary condition
on the right side of the cylinder, we have 
\begin{equation*}
f_{+}(t,\mathbf{x};\mathbf{v})=\int_{u_{x}\leqslant W\left( t\right)
}K\left( \mathbf{v}-\mathbf{i}W\left( t\right) ;\mathbf{u}-\mathbf{i}W\left(
t\right) \right) f_{-}(t,\mathbf{x};\mathbf{u})d\mathbf{u}.
\end{equation*}
Writing 
\begin{equation*}
f_{-}(t,\mathbf{x};\mathbf{u})=f_{+}(\tau ,\mathbf{\xi };\mathbf{u})\chi
_{1}(t,\mathbf{u})+f_{0}\left( \mathbf{u}\right) \chi _{0}(t,\mathbf{u}),
\end{equation*}%
we have 
\begin{eqnarray*}
f_{+}(t,\mathbf{x};\mathbf{v}) &=&b(v_{\bot })\int_{u_{x}\leqslant W\left(
t\right) }k(v_{x}-W(t),u_{x}-W(t)) \\
&&\times \left\{ a_{+}(\tau ,\mathbf{\xi ;}u_{x})b(u_{\bot })\chi _{1}(t,%
\mathbf{u})+a_{0}(u_{x})b(u_{\bot })\chi _{0}(t,\mathbf{u})\right\} d\mathbf{%
u}.
\end{eqnarray*}%
We then estimate 
\begin{eqnarray*}
f_{+}(t,\mathbf{x};\mathbf{v}) &\leqslant &b(v_{\bot
})\int_{\inf_{s<t}\left\langle W\right\rangle _{s,t}}^{W\left( t\right)
}k(v_{x}-W(t),u_{x}-W(t))a_{+}(\tau ,\mathbf{\xi ;}u_{x})du_{x} \\
&&+b(v_{\bot })\int_{-\infty }^{W\left( t\right)
}k(v_{x}-W(t),u_{x}-W(t))a_{0}(u_{x})du_{x}
\end{eqnarray*}
The last integral is uniformly bounded due to Assumption A4. For the other
integral we need an estimate on the size of $W\left( t\right)
-\inf_{s<t}\left\langle W\right\rangle _{s,t}$. With (\ref%
{estimate:inf_wW_st}) and%
\begin{equation*}
W\left( t\right) \leqslant V_{\infty }-\gamma h(t)\leqslant V_{\infty
}+\gamma e^{-B_{\infty }t},
\end{equation*}%
we get%
\begin{equation}
W\left( t\right) -\inf_{s<t}\left\langle W\right\rangle _{s,t}\leqslant
\gamma e^{-B_{\infty }t}+\frac{C\gamma ^{p+1}A_{+}}{\left\langle
t\right\rangle }.  \label{estimate:W(t)-inf}
\end{equation}%
Thus we can estimate%
\begin{equation*}
b(v_{\bot })\int_{\inf_{s<t}\left\langle W\right\rangle _{s,t}}^{W\left(
t\right) }k(v_{x}-W(t),u_{x}-W(t))a_{+}(\tau ,\mathbf{\xi ;}%
u_{x})du_{x}\leqslant b(v_{\bot })C\gamma a_{+}^{\ast },
\end{equation*}%
where $a_{+}^{\ast }$ is as in (\ref{def:a+*}), so that we have $b(v_{\bot
})a_{+}^{\ast }\leqslant b(v_{\bot })C\gamma a_{+}^{\ast }+Cb(v_{\bot })$.
That is, 
\begin{equation*}
b(v_{\bot })a_{+}^{\ast }\leqslant Cb(v_{\bot })
\end{equation*}%
for small enough $\gamma $. With this upper bound of $f_{+}(\tau ,\mathbf{%
\xi },u)$ we arrive at 
\begin{eqnarray*}
\left\vert r_{W}^{R}\left( t\right) \right\vert &=&\left\vert \int_{\partial
\Omega _{R}(t)}dS_{x}\int_{u_{x}\leqslant W\left( t\right) }d\mathbf{u}\
\ell (u_{x}-W(t))\left\{ f_{-}(t,\mathbf{x};\mathbf{u})-f_{0}(\mathbf{u}%
)\right\} \right\vert \\
&=&\left\vert \int_{\partial \Omega _{R}(t)}dS_{x}\int_{u_{x}\leqslant
W\left( t\right) }d\mathbf{u}\ \ell (u_{x}-W(t))\left\{ f_{+}(\tau ,\mathbf{%
\xi };\mathbf{u})\chi _{1}\left( t,\mathbf{u}\right) +f_{0}(\mathbf{u})\chi
_{0}(t,\mathbf{u})-f_{0}(\mathbf{u})\right\} \right\vert \\
&=&\left\vert \int_{\partial \Omega
_{R}(t)}dS_{x}\int_{\inf_{s<t}\left\langle W\right\rangle _{s,t}}^{W(t)}d%
\mathbf{u}\ \ell (u_{x}-W(t))\left\{ f_{+}(\tau ,\mathbf{\xi };\mathbf{u}%
)-f_{0}(\mathbf{u})\right\} \right\vert \\
&\leqslant &\int_{\partial \Omega _{R}(t)}dS_{x}\int_{\inf_{s<t}\left\langle
W\right\rangle _{s,t}}^{W(t)}d\mathbf{u}\ \ell (u_{x}-W(t))\left( f_{+}(\tau
,\mathbf{\xi };\mathbf{u})+f_{0}(\mathbf{u})\right) \\
&\leqslant &C\int_{\left\vert u_{\bot }\right\vert \leqslant \frac{2r}{%
t-\tau }}du_{\bot }\int_{\inf_{s<t}\left\langle W\right\rangle
_{s,t}}^{W(t)}du_{x}\ \ell (u_{x}-W(t))b(u_{\bot }).
\end{eqnarray*}%
Split the integral at $\tau =t/2$. As distinguished from the irreversal
case, the $\tau <\frac{t}{2}$ part provides the dominant contribution this
time. We have%
\begin{equation*}
I=C\int_{\inf_{s<t}\left\langle W\right\rangle _{s,t}}^{W(t)}du_{x}\
\int_{\left\vert u_{\bot }\right\vert \leqslant \frac{2r}{t-\tau },\tau <%
\frac{t}{2}}du_{\bot }\ell (u_{x}-W(t))b(u_{\bot })\leqslant C\frac{\left(
W(t)-\inf_{s<t}\left\langle W\right\rangle _{s,t}\right) ^{p+1}}{\left(
1+t\right) ^{d-1}}
\end{equation*}
Plugging (\ref{estimate:W(t)-inf}) into this expression, we have:%
\begin{equation*}
I\leqslant \frac{C}{\left( 1+t\right) ^{d-1}}\left( \gamma e^{-B_{\infty }t}+%
\frac{C\gamma ^{p+1}A_{+}}{\left\langle t\right\rangle }\right) ^{p+1}
\end{equation*}%
for $t\geqslant t_{0}$.

On the other hand, for the $\tau \in \left[ \frac{t}{2},t\right] $ part, we
know that 
\begin{eqnarray*}
u_{x} &=&\left\langle W\right\rangle _{\tau ,t}\ \geqslant \ V_{\infty
}-\gamma \left\langle g\right\rangle _{\tau ,t} \\
&\geqslant &V_{\infty }+\inf_{\frac{t}{2}\leqslant \tau \leqslant t}\frac{%
\gamma }{t-\tau }\int_{\tau }^{t}\left( e^{-B_{0}r}-\frac{\gamma ^{p}A_{+}}{%
\left\langle r\right\rangle ^{p+d}}\right) dr \\
&\geqslant &V_{\infty }+\inf_{\frac{t}{2}\leqslant \tau \leqslant t}\frac{%
\gamma }{t-\tau }\int_{\tau }^{t}\left( -\frac{\gamma ^{p}A_{+}}{%
\left\langle r\right\rangle ^{p+d}}\right) dr\ \geqslant \ V_{\infty }-C%
\frac{\gamma ^{p+1}A_{+}}{\left\langle t\right\rangle ^{p+d}},
\end{eqnarray*}
which yields%
\begin{eqnarray*}
II &=&C\int_{\inf_{s<t}\left\langle W\right\rangle _{s,t}}^{W(t)}du_{x}\
\int_{\left\vert u_{\bot }\right\vert \leqslant \frac{2r}{t-\tau },\frac{t}{2%
}\leqslant \tau \leqslant t}du_{\bot }\ell (u_{x}-W(t))b(u_{\bot }) \\
&\leqslant &C\int_{V_{\infty }-C\frac{\gamma ^{p+1}A_{+}}{\left\langle
t\right\rangle ^{p+d}}}^{W(t)}\ell (u_{x}-W(t))du_{x} \\
&\leqslant &C\left\vert W(t)-V_{\infty }-C\frac{\gamma ^{p+1}A_{+}}{%
\left\langle t\right\rangle ^{p+d}}\right\vert ^{p+1}\ .
\end{eqnarray*}
By (\ref{r:g and h}), we estimate%
\begin{eqnarray*}
&&\left\vert W(t)-V_{\infty }-C\frac{\gamma ^{p+1}A_{+}}{\left\langle
t\right\rangle ^{p+d}}\right\vert \\
&\leqslant &\max \left\{ \gamma e^{-B_{\infty }t},\gamma e^{-B_{0}t}\right\}
+\max \left\{ \frac{\gamma ^{p+1}A_{+}}{t^{p+d}}\chi \left\{ t\geqslant
t_{0}+1\right\} ,\frac{\gamma ^{p+1}A_{+}}{\left\langle t\right\rangle ^{p+d}%
}\right\} +C\frac{\gamma ^{p+1}A_{+}}{\left\langle t\right\rangle ^{p+d}} \\
&\leqslant &C\left( \gamma e^{-B_{\infty }t}+\frac{\gamma ^{p+1}A_{+}}{%
\left\langle t\right\rangle ^{p+d}}\right) .
\end{eqnarray*}%
where in the last line we used the fact that $B_{\infty }\leqslant B_{0}$
and $A_{-}\leqslant A_{+}$. Indeed, for large $t$, we have 
\begin{equation*}
-g(t,\gamma )\sim -\frac{\gamma ^{p}A_{+}}{\left\langle t\right\rangle ^{p+d}%
}\text{ and }-h(t,\gamma )\sim -\frac{\gamma ^{p}A_{-}}{\left\langle
t\right\rangle ^{p+d}}
\end{equation*}%
so that we require $A_{+}\geqslant A_{-}$ to make sure that $V_{\infty
}-\gamma g(t,\gamma )\leqslant V_{\infty }-\gamma h(t,\gamma )$.

With the estimates for $I$ and $II$ in hand, we have%
\begin{eqnarray}
\left\vert r_{W}^{R}\left( t\right) \right\vert &\leqslant &\frac{C}{\left(
1+t\right) ^{d-1}}\left( \gamma e^{-B_{\infty }t}+\frac{C\gamma ^{p+1}A_{+}}{%
\left\langle t\right\rangle }\right) ^{p+1}\chi \left\{ t\geqslant
t_{0}\right\}  \label{R:UpperRight1} \\
&&+C\left( \gamma e^{-B_{\infty }t}+\frac{\gamma ^{p+1}A_{+}}{\left\langle
t\right\rangle ^{p+d}}\right) ^{p+1}\chi \left\{ t\geqslant t_{0}\right\} . 
\notag
\end{eqnarray}%
Because on its face estimate (\ref{R:UpperRight1}) alone is not enough to
deduce for the reversal case that 
\begin{equation*}
r_{W}^{L}\left( t\right) +r_{W}^{R}\left( t\right) \geqslant 0,
\end{equation*}%
we now give a more precise bound. Let us recall from Definition \ref{def:W-R}
that $t_{0}=K_{0}\left\vert \log \gamma \right\vert $ with $\frac{1}{B_{0}}%
\leqslant K_{0}\leqslant \frac{2}{B_{0}}.$ The key here is to choose $\gamma 
$ small enough so that the quantity $B_{\infty }K_{0}$ can be bounded below
by 
\begin{equation*}
B_{\infty }K_{0}\geqslant \frac{B_{\infty }}{B_{0}}=\frac{\min_{V\in \left[
V_{\infty }-\gamma ,V_{\infty }+\gamma \right] }F_{0}^{\prime }(V)}{%
\max_{V\in \left[ V_{\infty }-\gamma ,V_{\infty }+\gamma \right]
}F_{0}^{\prime }(V)}\geqslant 1-\varepsilon \text{, }\varepsilon >0
\end{equation*}%
which can be made close to $1$. For $t\geqslant t_{0}$, we write 
\begin{eqnarray*}
e^{-B_{\infty }t} &=&e^{-(1-\varepsilon )B_{\infty }t}e^{-\varepsilon
B_{\infty }t}\leqslant e^{-(1-\varepsilon )B_{\infty }t}e^{-\varepsilon
B_{\infty }t_{0}} \\
&\leqslant &e^{-(1-\varepsilon )B_{\infty }t}e^{\varepsilon B_{\infty
}K_{0}\log \gamma }=e^{-(1-\varepsilon )B_{\infty }t}\gamma ^{\varepsilon
B_{\infty }K_{0}}\leqslant e^{-(1-\varepsilon )B_{\infty }t}\gamma ^{\frac{%
\varepsilon }{2}},
\end{eqnarray*}%
because $\gamma <1.$ By the choice $\varepsilon \in \left( 0,2p\right) $ and
because $A_{+}>0$, estimate (\ref{R:UpperRight1}) becomes%
\begin{eqnarray*}
\left\vert r_{W}^{R}\left( t\right) \right\vert &\leqslant &\frac{C}{\left(
1+t\right) ^{d-1}}\left( \gamma e^{-(1-\varepsilon )B_{\infty }t}\gamma ^{%
\frac{\varepsilon }{2}}+\frac{C\gamma ^{p+1}A_{+}}{\left\langle
t\right\rangle }\right) ^{p+1}\chi \left\{ t\geqslant t_{0}\right\} \\
&&+C\left( \gamma e^{-(1-\varepsilon )B_{\infty }t}\gamma ^{\frac{%
\varepsilon }{2}}+\frac{\gamma ^{p+1}A_{+}}{\left\langle t\right\rangle
^{p+d}}\right) ^{p+1}\chi \left\{ t\geqslant t_{0}\right\} \\
&\leqslant &\frac{C\gamma ^{\left( 1+\frac{\varepsilon }{2}\right) \left(
p+1\right) }A_{+}^{p+1}}{\left\langle t\right\rangle ^{p+d}}\chi \left\{
t\geqslant t_{0}\right\} +\frac{C\gamma ^{\left( 1+\frac{\varepsilon }{2}%
\right) \left( p+1\right) }A_{+}^{p+1}}{\left\langle t\right\rangle ^{\left(
p+d\right) \left( p+1\right) }}\chi \left\{ t\geqslant t_{0}\right\} .
\end{eqnarray*}%
Since $p>0$, the second term is absorbed into the first term for large
enough $t_{0}$, so that%
\begin{equation*}
\left\vert r_{W}^{R}\left( t\right) \right\vert \leqslant \frac{C\gamma
^{\left( 1+\frac{\varepsilon }{2}\right) \left( p+1\right) }A_{+}^{p+1}}{%
\left\langle t\right\rangle ^{p+d}}\chi \left\{ t\geqslant t_{0}\right\}
\end{equation*}%
as claimed.
\end{proof}

Putting together Lemmas \ref{Lemma:Upper and lower bound of R+} and \ref%
{Lemma:UpperBoundOfR-}, we conclude that 
\begin{equation*}
R_{W}(t)\geqslant \left( \frac{c_{1}\gamma ^{p+1}}{t^{p+d}}-\frac{%
C_{2}\gamma ^{\left( 1+\varepsilon \right) \left( p+1\right) }A_{+}^{p+1}}{%
t^{p+d}}\right) \chi \left\{ t\geqslant t_{0}\right\}
\end{equation*}%
and%
\begin{equation*}
R_{W}(t)\leqslant \frac{C_{1}\left( \gamma +\gamma ^{p+1}A_{+}\right) ^{p+1}%
}{\left( 1+t\right) ^{p+d}}+\frac{C_{2}\gamma ^{\left( 1+\varepsilon \right)
\left( p+1\right) }A_{+}^{p+1}}{t^{p+d}}\chi \left\{ t\geqslant
t_{0}\right\} .
\end{equation*}%
Theorem \ref{Thm:REstimate} then follows since $\gamma $ is small.


\end{document}